\documentclass[11pt,reqno,a4paper]{amsart}

\usepackage{mathrsfs, graphics}
\usepackage{amssymb}
\usepackage[notref,notcite,final]{showkeys}
\usepackage{hyperref}\hypersetup{colorlinks=true, citecolor=blue}
\usepackage{graphicx, epstopdf}\graphicspath{{./immaginiprova/}}
\usepackage{bbm}
\numberwithin{equation}{section}

\newtheoremstyle{mytheorem}
{}
{}
{\it}
{\parindent}
{\bf}
{.}
{ }
{\thmnumber{#2.~}\thmname{#1}\thmnote{~\rm#3}}

\newtheoremstyle{myremark}
{}
{}
{\rm}
{\parindent}
{\bf}
{.}
{ }
{\thmnumber{#2.~}\thmname{#1}\thmnote{~\rm#3}}

\newtheoremstyle{myparagraph}
{}
{}
{\rm}
{\parindent}
{\bf}
{.}
{ }
{\thmnumber{#2.~}\thmname{#1}\thmnote{#3}}

\theoremstyle{mytheorem}

\newtheorem{theorem}[subsection]{Theorem}
\newtheorem{definition}[subsection]{Definition}
\newtheorem{lemma}[subsection]{Lemma}
\newtheorem{corollary}[subsection]{Corollary}
\newtheorem{proposition}[subsection]{Proposition}

\theoremstyle{myremark}
\newtheorem{remark}[subsection]{Remark}

\theoremstyle{myparagraph}

\newtheorem*{parag*}{}

\makeatletter
\def\@secnumfont{\sc}
\def\section{\@startsection{section}{1}%
\z@{1.5\linespacing\@plus .2\linespacing}{.7\linespacing}%
{\normalfont\sc\centering}}
\makeatother

\makeatletter
\def\ps@headings{\ps@empty
 \def\@evenhead{%
  \setTrue{runhead}%
  \normalfont\footnotesize
  \rlap{\thepage}\hfil
  \def\thanks{\protect\thanks@warning}%
  \leftmark{}{}\hfil}%
 \def\@oddhead{%
  \setTrue{runhead}%
  \normalfont\footnotesize\hfil
  \def\thanks{\protect\thanks@warning}%
  \rightmark{}{}\hfil \llap{\thepage}}%
\let\@mkboth\markboth}
\makeatother
\makeatletter
\renewenvironment{proof}[1][\proofname]{\par
  \pushQED{\qed}%
  \normalfont \topsep6\p@\@plus6\p@\relax
  \trivlist
  \itemindent\normalparindent
  \item[\hskip\labelsep
    \bfseries
    #1\@addpunct{.}]\ignorespaces
}{%
  \popQED\endtrivlist\@endpefalse
}
\providecommand{\proofname}{Proof}
\makeatother
\newcommand{\Mass}{\mathbb{M}}
\newcommand{\TP}{\textbf{TP}}
\newcommand{\OTP}{\textbf{OTP}}
\newcommand{\weak}{\overset{*}\rightharpoonup}

\newcommand{\R}{\mathbb{R}}

\newcommand{\E}{\mathcal{E}^\alpha}
\newcommand{\MM}{\mathbb{M}^\alpha}
\newcommand{\N}{\mathbb{N}}

\newcommand{\Haus}{\mathscr{H}}
\newcommand{\Leb}{\mathscr{L}}

\newcommand{\M}{\mathscr{M}_+}
\newcommand{\PP}{\mathscr{P}}
\newcommand{\PPP}{\mathbf{P}}
\newcommand{\QQQ}{\mathbf{Q}}

\newcommand{\conc}{{conc}}

\newcommand{\Lip}{\mathrm{Lip}_1}

\newcommand{\supp}{\mathrm{supp}}

\newcommand{\e}{\varepsilon}
\newcommand{\dV}{d_V\kern-1pt}

\newcommand{\trait}[3]{\vrule width #1ex height #2ex depth #3ex}

\newcommand{\trace}{\mathchoice%
  {\mathbin{\trait{.12}{1.2}{.03}\trait{.8}{0.09}{0.03}}}
  {\mathbin{\trait{.12}{1.2}{.03}\trait{.8}{0.09}{0.03}}}
  {\mathbin{\hskip.15ex\trait{.09}{.84}{0.02}\trait{.56}{.07}{.02}}\hskip.15ex}
  {\mathbin{\trait{.07}{.6}{.01}\trait{.4}{.06}{.01}}}}

\makeatletter
\@addtoreset{footnote}{section}
\makeatother

\begin{document}

	%
\pagestyle{empty}
\pagestyle{myheadings}
\markboth%
{\underline{\centerline{\hfill\footnotesize%
\textsc{Maria Colombo, Antonio De Rosa and Andrea Marchese}%
\vphantom{,}\hfill}}}%
{\underline{\centerline{\hfill\footnotesize%
\textsc{Stability for the mailing problem}%
\vphantom{,}\hfill}}}

	%
\thispagestyle{empty}

~\vskip -1.1 cm

	%

\vspace{1.7 cm}

	%
{\large\bf\centering
Stability for the mailing problem\\
}

\vspace{.6 cm}

	%
\centerline{\sc Maria Colombo, Antonio De Rosa and Andrea Marchese}

\vspace{.8 cm}

{\rightskip 1 cm
\leftskip 1 cm
\parindent 0 pt
\footnotesize

	%
{\sc Abstract.}
We prove that optimal traffic plans for the mailing problem in $\R^d$ are stable with respect to variations of the given coupling, above the critical exponent $\alpha=1-1/d$, thus solving an open problem stated in the book \emph{Optimal transportation networks}, by Bernot, Caselles and Morel.
We apply our novel result to study some regularity properties of the minimizers of the mailing problem. In particular, we show that only finitely many connected components of an optimal traffic plan meet together at any branching point.

\par
}


\section{Introduction}
The optimal branched transportation is a variant of the Monge-Kantorovich optimal transportation of mass, in which the mass particles are assumed to interact (rather than traveling independently) while moving from a source to a target distribution. In particular there is a gain in the cost of the transportation whenever some mass is transported in a grouped way. A consequence of this assumption is that the particles' paths form a one-dimensional network which develops branched structures. 

The interest in branched transportation arises from the observation of common structures in many natural supply-demand systems, such as the nerves of a leaf, river basins, or the nervous, the bronchial, and the cardiovascular system. It has been also used to model several human-designed systems, including power supply, urban planning, and irrigation. 

In order to translate in mathematical terms the convenience of grouping mass during the transportation, one considers a cost functional obtained integrating along the network created by the particles' trajectories a subadditive function of the intensity of the flow. 
The first \emph{discrete} model was introduced by Gilbert in \cite{Gilbert}. Then it was extended by Xia to a \emph{continuous} framework adopting an Eulerian formulation (describing the particles' flow) which uses Radon vector-valued measures, or, equivalently, 1-dimensional currents, called in this framework \emph{traffic paths}, (see \cite{Xia}). At about the same time a Lagrangian formulation (describing the particles' trajectories) of the continuous version of the problem was proposed by Maddalena, Morel, and Solimini \cite{MSM}, using the notion of \emph{traffic plans}, i.e. measures on the set of Lipschitz curves.

A great interest has been devoted to branched transportation problems in the last years, with results concerning existence  \cite{Xia,MSM,BCM1,BeCaMo,brabutsan,Pegon,flat-relax}, regularity \cite{xia2,MR2250166,DevSol,DevSolElementary,morsant,xiaBoundary,BraSol,brawir} and strategies to compute minimizers \cite{OuSan,BCF,BLS,massoubo,BOO,marmass1,marmass,gol,CDRM,MMT}.

In the present paper we restrict our attention to the so called \emph{mailing problem} (also known as \emph{who goes where problem}): a version of the branched transportation problem in which not only the initial and target mass distributions are given, but also a coupling between the two measures: in other words one knows a priori where each mass particle should be moved, and the only unknown is an optimal transportation network realizing such coupling. 

More precisely we address an open problem on the stability of optimal traffic plans for the mailing problem, with respect to variations of the given mass distribution, raised in \cite[Remark 6.13]{BCM}. Let us introduce some preliminary notation in order to state the problem and our main result.\\

Let $X:=\overline{B(0,R)}\subset\R^d$ and fix a probability measure $\pi \in \PP(X\times X)$. A traffic plan is a probability measure $\PPP \in \PP(\Lip)$ on the space of $1$-Lipschitz curves $\gamma:[0,\infty)\to X$, supported on curves which are eventually constant. The topology on $\Lip$ is induced by the uniform convergence on compact subsets of $[0,\infty)$. We say that $\PPP$ is associated to the coupling $\pi$ if it satisfies
$$(e_0,e_\infty)_\sharp \PPP=\pi, \quad \text{where} \qquad e_0(\gamma):= \gamma(0) \, \, \text{ and } \, \, \quad e_\infty(\gamma):= \lim_{t \to \infty}\gamma(t).$$ 
For given $\alpha\in(0,1)$, the \emph{$\alpha$-energy} of $\PPP$ is defined by 
\begin{equation*}
\E(\PPP):= \int_{\Lip} \int_{\R^+} |\gamma( t)|_{\PPP}^{\alpha-1}|\dot \gamma(t)| dt\, d\PPP(\gamma),
\end{equation*}
where the multiplicity at a point $x$ is given by
$$|x|_\PPP := \PPP (\{\gamma \in \Lip :  \, \gamma(t) = x, \quad \mbox{for some $t$}\}).$$
We say that a traffic plan $\PPP$, with a coupling $\pi$, is optimal, and we write $\PPP\in\OTP(\pi)$, if 
$$\E(\PPP) \leq \E(\PPP'), \qquad \text{for every $\PPP'$ such that $(e_0,e_\infty)_\sharp \PPP'=\pi$} .$$
On optimal traffic plans, the $\alpha$-energy coincides with the \emph{$\alpha$-mass}
$$\Mass^\alpha(\PPP):=\int_{\R^d}|x|_\PPP^\alpha d\Haus^1,$$
recovering the desired notion of cost, i.e. a quantity which is computed integrating a subadditive function of the mass flow.

In the class of traffic plans, we consider the usual notion of weak$^*$ convergence of Radon measures. To get continuity of the map $(e_\infty)_\sharp$ on traffic plans, with respect to the weak$^*$ convergence, we should require a technical assumption on the class of traffic plans that we consider, namely that, denoting $T(\gamma)$ the stopping time of a curve $\gamma$, there exists a constant $C$ such that each traffic plan $\PPP$ in the class satisfies
$$\int_{\Lip}T(\gamma)d\PPP\leq C.$$
In this case we say that $\PPP\in \TP_C$. Notice that this is a tightness condition on a class of probability measures. The main result of the paper is the following:
\begin{theorem}\label{thm:main}
	Let $\alpha>1-\frac{1}{d}$ and $C>0$. Let $\pi \in \PP(X\times X)$ and $\{ \pi_n\}_{n\in\N}\subseteq \PP(X\times X)$ be such that
	\begin{equation}\label{hp:supp-n-convergence}
\pi _n \overset{*}\rightharpoonup \pi, \qquad \mbox{weakly$^*$ in the sense of measures}.
	\end{equation}
	For every $n\in \N$, let $\PPP_n\in \OTP(\pi_n)$ be an optimal traffic plan for the mailing problem and assume that $\PPP_n\in\TP_C$ and that there exists a traffic plan $\PPP\in \PP(\Lip)$ satisfying
	$$
	\PPP_n\overset{*}\rightharpoonup\PPP,\qquad \mbox{ as $n\to\infty$}.
	$$
	Then $\PPP$ is optimal, namely $\PPP\in \OTP(\pi)$ and moreover $\E(\PPP_n)\to \E(\PPP)$. 
\end{theorem}

\begin{remark}
A comment about the $\TP_C$ assumption in Theorem \ref{thm:main} is necessary. In \cite[Proposition 6.12]{BCM} the assumption $\PPP_n\in\TP_C$ is not explicitly stated, but it is tacitly used. In \cite{CDRM2}, we show that without such assumption the stability as well as the convergence of $(
{e_\infty})_\sharp\PPP_n$ to $({e_\infty})_\sharp\PPP$ could fail, both for the mailing problem and for the ``standard'' branched transportation problem.
\end{remark}
 
Besides the obvious relevance for numerical applications, Theorem \ref{thm:main} can be used to carry on part of the regularity program for optimal traffic plans for the mailing problem, which is another open question in \cite[Section 15.3]{BCM}.
In Theorem \ref{rego2} we prove that any branching point of an optimal traffic plan splits the support in only finitely many connected components and that the number of connected components of an optimal traffic plan is finite.

\section{Notation and preliminaries}\label{s:notation}
We denote by $\R^{d}$ the $d$-dimensional Euclidean space and by $B(x,r)$ the open ball ${\{y\in\R^{d}:|x-y|<r\}}$. For a set $A\subset \R^d$, we denote by $\overline{A}$ its closure and by $A^c:=\R^d\setminus A$ its complementary set.
\subsection{Measures and rectifiable sets}
For a locally compact, separable metric space $Y$, we denote by $\M(Y)$ the set of positive finite Radon measures in $Y$, namely the set of positive measures on the $\sigma$-algebra of Borel sets of $Y$ that are finite and inner regular. 
The set of probability measures, i.e. those measures $\mu$ satisfying $\mu(Y)=1$, is denoted $\PP(Y)$.

 For $\mu,\nu\in\M(Y)$, we write $\mu\leq \nu$ if $\mu(A)\leq \nu(A)$ for every Borel set $A$.
For a measure \(\mu\) we denote by 
$$\supp (\mu):= \bigcap\{C\subset Y:C \mbox{ is closed and } \mu(Y \setminus C)=0\}$$
its \emph{support}. We say that $\mu$ is \emph{supported} on a Borel set $E$ if $\mu(Y\setminus E)=0$.
For a  Borel set \(E\),  \(\mu\trace E\) is the  restriction of \(\mu\) to \(E\), i.e. the measure defined by 
$$[\mu\trace E](A):=\mu(E\cap A), \qquad \mbox{for every Borel set $A$.}$$ 
We denote by $|\mu|$ the total variation (or mass) of $\mu$, i.e. $|\mu|=\mu(Y)$.

For a measure $\mu\in\M(Y)$, a metric space $Z$, and a Borel map \(\eta :Y\to Z\), we let \({\eta_\sharp \mu\in\M(Z)}\) be the push-forward measure, namely 
$$\eta_\sharp \mu(A):=\mu(\eta^{-1}(A)), \qquad \mbox{for every Borel set }A\subset Z.$$

We use $\Leb^d$ and $\Haus^k$ to denote respectively the $d$-dimensional Lebesgue measure on $\R^d$ and the $k$-dimensional Hausdorff measure (see \cite{SimonLN}).


A  set \(K\subset \R^d\) is said \emph{countably \(k\)-rectifiable} (or simply \emph{\(k\)-rectifiable}) if it can be covered, up to an \(\Haus^k\)-negligible set, by countably many $k$-dimensional submanifolds of class \(C^1\).

\subsection{Lipschitz curves}
Let $X \subseteq \R^d$ be a compact, convex 
set. We denote by $\Lip$ the space of $1$-Lipschitz curves $\gamma: \R^+ \to X$, endowed with the metric of uniform convergence on compact subsets of $\R^+:=\{t\in \R: t\geq 0\}$.

For every $\gamma \in \Lip$, we denote its (possibly infinite) {\emph{stopping time}} by
$$T(\gamma):=\inf\{t\in \R^+ : \dot \gamma(s)=0, \quad \mbox{for a.e. \, $s\geq t$}\}.$$
On the set $\{ \gamma \in \Lip(X) : T(\gamma)<\infty\}$, we define the evaluations at time $0$ and ``at time $\infty$''
$$e_0(\gamma):= \gamma(0), \qquad 
e_\infty(\gamma):= \gamma(T(\gamma)).
$$

\subsection{Traffic plans}
In this section, we recall the basic definitions concerning traffic plans (see \cite{BCM} for an exhaustive description). 
%
%
%
%

We denote with $\TP(X)$ the space of \emph{traffic plans} with values in $X$, i.e.
$$\TP(X):=\{\PPP\in \M(\Lip(X)): \PPP \mbox{ is supported on } \{ \gamma: T(\gamma)<\infty\}\}.$$

Every traffic plan $\PPP$ naturally identifies a \emph{coupling} between an ``irrigating'' and an ``irrigated'' probability measure on $\R^d$. We consider the measure in $\M(\R^d\times \R^d)$:
$$\pi_\PPP:=(e_0,e_\infty)_\sharp \PPP,$$
called the \emph{coupling} of $\PPP$ and, denoting $e_1$, $e_2$ the projections from $\R^d\times \R^d$ onto the first and the second components, we call the \emph{irrigating} and \emph{irrigated measure} of  $\PPP\in  \TP(X)$ respectively \begin{equation}\label{def:bordo}
	\mu^-_\PPP := (e_1)_\sharp \pi_\PPP = (e_0)_\sharp \PPP, \qquad  \mu^+_\PPP := (e_2)_\sharp \pi_\PPP = (e_\infty)_\sharp \PPP.
\end{equation}
  For every $\pi \in \M(X\times X)$, we denote the set of traffic plans with coupling $\pi$ by 
$$\TP (\pi):=\{\PPP \in \TP(X)\mbox{: $\pi_\PPP=\pi$} \}.$$
Sometimes we will not fix a coupling of a traffic plan, but just the pair of the initial and the final measure of the transportation process. Given two measures $\mu^-,\mu^+ \in \M(X)$, we denote the set of traffic plans with marginals $\mu^-,\mu^+$ by 
$$\TP(\mu^-,\mu^+):=\{\PPP \in \TP(X)\mbox{: $\mu^-_\PPP=\mu^-$ and $\mu^+_\PPP=\mu^+$} \}.$$

\subsection{Energies of traffic plans}
For a traffic plan $\PPP$ and a point $x\in \R^d$, we define the \emph{multiplicity} of $\PPP\in  \TP(X)$ at $x$ by 
$$|x|_\PPP := \PPP (\{\gamma \in \Lip(X) :  \, \gamma(t) = x, \quad \mbox{for some $t\in\R^+$}\}).$$
Roughly speaking, the multiplicity measures the amount of paths passing through a given point $x$, without counting repetitions or orientations.
For any $\alpha \in [0, 1)$, we define a functional $\E: \TP(X) \to [0,\infty]$, that we call \emph{$\alpha$-energy}, defined by 
\begin{equation}\label{energy}
\E(\PPP):= \int_{\Lip} \int_{\R^+} |\gamma( t)|_{\PPP}^{\alpha-1}|\dot \gamma(t)| dt\, d\PPP(\gamma),
\end{equation}
where we used the convention that $0^{\alpha-1} = +\infty$ for $\alpha \in [0,1).$

%
%
%

We denote the set of optimal traffic plans with coupling $\pi \in \M(\R^d\times \R^d)$ as 
$$\OTP (\pi):=\{\PPP \in \TP (\pi) : \E(\PPP)\leq \E(\PPP') \quad \forall \,  \PPP' \in \TP (\pi) \}.$$

The existence of an optimal traffic plan follows from the lower semicontinuity of the $\alpha$-energy, which can be found in \cite{BCM} (see also \cite{Pegon}). The lower semicontinuity holds only in a certain subclass of $\TP(X)$. For every $C>0$ we define the class $\TP_C(X)$ as the traffic plans $\PPP\in\TP(X)$ such that
\begin{equation}\label{eqn:TPC}
	\int_{\Lip} T(\gamma)\, d\PPP \leq C.
\end{equation}

\begin{lemma}\label{TPC}
Let $C>0$ and let $\{\PPP_n\}_{n \in \N}$ be a sequence of traffic plans in $\TP_C(\R^d)$ weakly$^{*}$ converging to a traffic plan $\PPP$. Then 
$$\liminf_{n\to \infty} \E(\PPP_n) \geq \E(\PPP).$$
\end{lemma} 

Let us denote $\E(\pi):=\E(\PPP)$ for some $\PPP \in \OTP (\pi)$. We recall the main result on the structure of traffic plans with finite $\alpha$-energy:
\begin{theorem}\emph{(see \cite[Theorem 4.10]{BCM}).}\label{Theoret}
Let $\alpha\in [0,1)$ and $\PPP$ be a traffic plan with $\E(\PPP)<\infty$.  Then there exists a rectifiable set $E$ such that
	$$\gamma(t)\in E, \qquad \mbox{for $\PPP$-a.e. $\gamma$ and for a.e. $t$.}$$
\end{theorem}
In the case $\E(\PPP)<\infty$, we can define an alternative notion of ``energy'', called $\alpha$-\emph{mass}. Let $E$ be the set associated to $\PPP$ as in Theorem \ref{Theoret}, we denote
\begin{equation}\label{mass}
\MM(\PPP):= \int_{E} |x|_\PPP^{\alpha} d\Haus^1(x).
\end{equation}
%
Moreover we recall the relationship between the functionals $\E$ and $\Mass^\alpha$.
%


\begin{proposition}\label{massaenergia}\emph{(see \cite[Proposition 4.8]{BCM}).}
	For every $\PPP \in \TP (\pi_{\PPP})$ it holds 
	\begin{equation}\label{energynonopt}
		\E(\pi_{\PPP})\geq \Mass^\alpha(\PPP).
	\end{equation}
	If moreover $\PPP\in \OTP (\pi_{\PPP})$, then 
	\begin{equation}\label{energyopt}
		\E(\PPP)=\Mass^\alpha(\PPP).
	\end{equation}
\end{proposition}

Optimal traffic plans satisfy a structural condition, usually called \emph{simple path property}, see \cite[Definition 4.7]{BCM}: 
\begin{definition}
For each $\gamma \in \Lip$, we define
$$D(\gamma) = \{x \in \R^d | \exists s,t \in [0,+\infty): s\neq t, \,\gamma(s)=\gamma(t)=x\}.$$
We say that a traffic plan $\PPP$ satisfies the simple path property if 
$$\int_{\Lip} \Haus^1(D(\gamma))\, d\PPP(\gamma) = 0.$$
\end{definition}
\begin{remark}\label{simple}
In \cite[Proposition 4.8]{BCM} it is proved that the conclusion \eqref{energyopt} holds even if $\PPP$ satisfies the simple path property (but it is not necessarily an optimum).
\end{remark}

Optimal traffic plans enjoy an important regularity property, usually called \emph{single path property}. Roughly speaking all fibers between two given points coincide (but they may have opposite orientations). See \cite[Proposition 7.4]{BCM} and the discussion before, for a formal definition.
%

The last prerequisite, crucial for our proof, is the fact that, when $\alpha>1-\frac{1}{d}$, the minimal transportation cost between two probability measures supported on a compact set is bounded by a constant which depends only on the diameter of the set.

\begin{theorem}\label{irrigability}\emph{(see \cite[Corollary 6.9]{BCM}).}
Let $\alpha > 1-\frac 1d$, and let $\mu^-, \mu^+ \in \M(X)$.
Then there exists $\PPP \in \TP(\mu^-,\mu^+)$ satisfying
$$\E(\PPP)\leq C_{\alpha,d}\,{\rm{diam}(X)}|\mu^-|,$$
where $C_{\alpha,d}$ is a constant depending only on $\alpha$ and $d$. 
\end{theorem}



\section{Concatenation and compatibility with the mailing problem}
\subsection{Disintegration}
%

Let $\PPP$ be any traffic plan, with coupling $\pi$. By the disintegration theorem for Radon measures, see \cite[Theorem 2.28]{AFP}, applied to $\PPP$, with respect to the map ${(e_0,e_\infty):\Lip\to X\times X}$ (which evaluates each curve at its initial and final point), we deduce that $\PPP$ can be written as 
$$
\PPP =\pi \otimes \PPP^{x,y},
$$
for a suitable family of probability measures $\PPP^{x,y}$ supported on the set of curves which begin at $x$ and end at $y$.
The latter means that 
$$\PPP(A):=\int_{X\times X}\PPP^{x,y}(\{\gamma\in A:(e_0,e_\infty)(\gamma)=(x,y)\})d\pi(x,y), \quad\mbox{ for every Borel set $A\subset\Lip$}.$$
This is useful to prove the following lemma:
\begin{lemma}\label{lemma:somma-plan}
Let $N \in \N$ and for every $i=1,..., N$ let $\mu^-_i,\mu^+_i \in \M(X)$ such that ${\mu^-_i(X) =\mu^+_i(X)}$. Let $\PPP \in \TP\big(\sum_{i=1}^N \mu^-_i \times \mu^+_i\big)$. Then for every $i=1,..., N$ there exists $\PPP^i \in \TP\big(\mu^-_i \times \mu^+_i\big)$ such that
\begin{equation}\label{eqn:chi-come-somma}
\PPP = \sum_{i=1}^N \PPP^i.
\end{equation}
\end{lemma}
\begin{proof}
We disintegrate  the traffic plan $\PPP$ with respect to $(e_0, e_\infty)$ to get a family of probability measures $\{\PPP^{x,y}\}_{(x,y)\in X \times X}$ such that
$$
\PPP = \Big(\sum_{i=1}^N \mu^-_i \times \mu^+_i\Big) \otimes \PPP^{x,y},
$$
namely for every test function $\varphi:\Lip(X)\to\R$ 
\begin{equation}
\label{eqn:defn-chi-conc}
\int \varphi(\gamma) \, d \PPP(\gamma) = \int \varphi( \gamma) \, d \PPP^{x,y}(\gamma)\, d\Big(\sum_{i=1}^N \mu^-_i \times \mu^+_i\Big)(x,y)=\sum_{i=1}^N \int \varphi( \gamma) \, d \PPP^{x,y}(\gamma)\, d(\mu^-_i \times \mu^+_i)(x,y).
\end{equation}
For every $i=1,..., N$ we define
$$
\PPP^i := \big(\mu^-_i \times \mu^+_i\big)\otimes \PPP^{x,y}.
$$
Then \eqref{eqn:chi-come-somma} is directly implied by \eqref{eqn:defn-chi-conc}.
Recalling that for every $(x,y) \in X \times X$ and for $\PPP^{x,y}$-a.e $\gamma$, it holds $(e_0(\gamma),e_\infty(\gamma))=(x,y)$, for every  test function $\psi:X\times X\to\R$ we compute
\begin{equation}
\begin{split}
\int \psi(x,y)\, d(e_0,e_\infty)_\# \PPP^i(x,y)&= \int \psi(e_0(\gamma),e_\infty(\gamma))\, d \PPP^i(\gamma)\\ 
&= \int \psi(e_0(\gamma),e_\infty(\gamma))\, d\PPP^{x,y}(\gamma)d\big(\mu^-_i \times \mu^+_i\big)(x,y)\\
&= \int \psi(x,y)\, d\big(\mu^-_i \times \mu^+_i\big)(x,y),
\end{split}
\end{equation} 
which reads $(e_0,e_\infty)_\# \PPP^i = \mu^-_i \times \mu^+_i$, i.e. $\PPP^i \in \TP\big(\mu^-_i \times \mu^+_i\big)$.
\end{proof}

\subsection{Concatenation}
In this paper we will often need to perform the following operation: given two traffic plans $\PPP_1\in\TP(\mu^-_1,\mu^+_1)$ and $\PPP_2\in\TP(\mu^-_2,\mu^+_2)$, with $\mu^+_1=\mu^-_2$, we want to define a traffic plan $\PPP\in\TP(\mu^-_1,\mu^+_2)$, obtained as a ``concatenation'' of the curves in $\PPP_1$ and those in $\PPP_2$. We begin by defining the concatenation of two curves.
\begin{definition}[(Concatenation of two curves)]
Given two curves $\gamma_1,\gamma_2\in \Lip$ such that $T(\gamma_1)<\infty$ and $e_\infty(\gamma_1)= e_0(\gamma_2)$, we define their \emph{concatenation} $c(\gamma_1, \gamma_2) \in \Lip$ as
\begin{equation*}
c(\gamma_1, \gamma_2)(t) := 
\begin{cases}
\gamma_1(t) \qquad &\mbox{if } t\leq T(\gamma_1)
\\
\gamma_2(t-T(\gamma_1)) \qquad &\mbox{if }t> T(\gamma_1).
\end{cases}
\end{equation*} 
We now extend the map $c$ to any pair $(\gamma_1, \gamma_2)$ of Lipschitz curves defining the concatenation map $conc: \Lip \times \Lip \to \Lip$ as follows:
\begin{equation*}
conc(\gamma_1, \gamma_2) := 
\begin{cases}
c(\gamma_1, \gamma_2) \qquad &\mbox{if  $T(\gamma_1)<\infty$ and $e_\infty(\gamma_1)= \gamma_2(0)$}
\\
0 \qquad &\mbox{otherwise},
\end{cases}
\end{equation*} 
where $0$ denotes the constant curve at the point $0$.
\end{definition}

The concatenation induces an analogous operation on traffic plans, which is in general a multi-map.
\begin{definition}[(Set of concatenations of two given traffic plans)]\label{defn:set-of-conc}
	Let $\mu^-, \nu, \mu^+ \in \M(X)$ with the same total mass. Let $\PPP_1 \in \TP(\mu^-, \nu)$ and $\PPP_2 \in \TP(\nu, \mu^+)$. A concatenation of $\PPP_1$ and $\PPP_2$ is any traffic plan of the form $\PPP = \conc_\# P$ which belongs to $\TP(\mu^-,\mu^+)$ and where $P \in \M(\Lip\times \Lip )$ satisfies
	
	\begin{enumerate}
	\item for $P$-a.e. $(\gamma_1, \gamma_2)$, we have that $\gamma_1$ is eventually constant and $e_\infty(\gamma_1)= \gamma_2(0)$.
	\item $(\rho_i)_\# P = \PPP_i$ for $i=1,2$, where $\rho_1, \rho_2 : \Lip \times \Lip \to \Lip$ denote the projections on the first and second  component, respectively.
\end{enumerate}
	
\end{definition}

We notice that the set of $P$'s that generate a concatenation of $\PPP_1$ and $\PPP_2$ as in Definition~\ref{defn:set-of-conc} is convex, since (1) and (2) are stable under convex combinations.

One may wonder why we give the definition of the set of concatenations instead of choosing a more canonical representative in this set (for instance, the one described in Lemma~\ref{lemma:all-you-want-on-conc}
(4) below). The reason is hidden in Lemma~\ref{lemma:all-you-want-on-conc}
(5) below: the concatenation, seen as a multi-map is stable by addition. This is not the case, in general, if one fixes a selection of the multi-map. 


The existence of a concatenation satisfying Definition~\ref{defn:set-of-conc} is trivial in some simple cases: for instance, if we disintegrate $\PPP_1$ and $\PPP_2$ as 
$$\PPP_i = \nu \otimes \PPP^x_i \qquad \mbox{ for } i=1,2,$$
with respect to $e_\infty$ and $e_0$ respectively, and if we assume that, for $\nu$-a.e. $x$, $\PPP^x_1$ (resp $\PPP^x_2$) is supported on a single curve ending in (resp. starting at) the point $x$, then the (unique) concatenation of $\PPP_1$ and $\PPP_2$ is obtained by concatenating for $\nu$-a.e. $x$ the unique curve in the support of $\PPP_1^x$ (that ends in $x$) with the unique curve in the support of $\PPP_2^x$ (that starts at $x$). 

When we drop the assumption that $\PPP^x_1 $ and $\PPP^x_2$ are Dirac deltas for $\nu$-a.e. $x$, the operation is more involved, and the set of all concatenations has infinitely many elements. In the following lemma, we summarize the properties of the set of concatenations. 
\begin{lemma}[(Properties of the set of concatenations)]\label{lemma:all-you-want-on-conc}
Let $\mu^-, \nu, \mu^+ \in \PP(X)$. Let $\PPP_1 \in \TP(\mu^-, \nu)$, $\PPP_2 \in \TP(\nu, \mu^+)$ and let $\PPP$ be any concatenation of $\PPP_1$ and $\PPP_2$. Then
\begin{enumerate}
	\item we have the following inequalities regarding the multiplicities
	\begin{equation}
	\label{ts:density-conc}
	\max\{ |z|_{\PPP_1}, |z|_{\PPP_2}\} \leq |z|_\PPP \leq |z|_{\PPP_1}+ |z|_{\PPP_2} \qquad \mbox{for every }z \in \R^d;
	\end{equation}
	\item for every $\phi:\R^d \mapsto [0,\infty)$ Borel function, setting $
	A_{\phi}(\gamma):= \int_{0}^{\infty} \phi(\gamma(t)) |\dot \gamma(t)|\, dt
	$, we have
	$$\int_{\Lip}A_{\phi}(\gamma) \, d \PPP(\gamma) = \int_{\Lip}A_{\phi}(\gamma) \, d \PPP_1(\gamma) + \int_{\Lip}A_{\phi}(\gamma) \, d \PPP_2(\gamma);
	$$
	\item we have the energy bounds
	\begin{equation}
	\label{ts:ts-conc}
	\E(\PPP) \leq \E(\PPP_1)+\E(\PPP_2), \qquad \MM(\PPP) \leq \MM(\PPP_1)+\MM(\PPP_2);
	\end{equation}
	\item the set of concatenations of $\PPP_1$ and $\PPP_2$ is nonempty;
	\item the set of concatenations is stable under addition: if $\tilde\PPP_1 \in \TP(\tilde\mu^-, \tilde\nu)$, $\tilde\PPP_2 \in \TP(\tilde\nu, \tilde\mu^+)$ and $\tilde\PPP$ is a concatenation of $\tilde\PPP_1$ and $\tilde\PPP_2$, then $\PPP+\tilde \PPP$ is a concatenation of $\PPP_1+\tilde\PPP_1$ and $\PPP_2+\tilde\PPP_2$.
\end{enumerate}
\end{lemma}
\begin{proof}
By Definition \ref{defn:set-of-conc}, there exists $P \in \PP(\Lip\times \Lip )$ such that
\begin{equation}
\label{eqn:defn-chi-P}
\int \varphi(\gamma) \, d \PPP(\gamma) = \int \varphi(conc(\gamma_1, \gamma_2)) \, dP(\gamma_1,\gamma_2).
\end{equation}
for every Borel function $\varphi: \Lip \to \R$.

We claim that $\PPP$ satisfies \eqref{ts:density-conc}.
Indeed, given $z \in \R^d$ we apply \eqref{eqn:defn-chi-P} with $\varphi(\gamma) = 1_{z\in Im(\gamma)}$ and we notice that $1_{z\in Im(conc(\gamma_1, \gamma_2))} = \max\{ 1_{z\in Im(\gamma_1)} , 1_{z\in Im(\gamma_2)} \}$ to get
$$
|z|_\PPP 
=
\int 1_{x\in Im(\gamma)} \, d \PPP(\gamma) = \int \max\{ 1_{x\in Im(\gamma_1)} , 1_{x\in Im(\gamma_2)} \} \, d P(\gamma_1,\gamma_2).
$$
We estimate
\begin{equation}
\begin{split}
|z|_\PPP 
&\geq 
\int  1_{x\in Im(\gamma_1)}  \, d P(\gamma_1,\gamma_2)=\int  1_{x\in Im(\gamma_1)}  \, d (\rho_1)_\#P(\gamma_1)=\int  1_{x\in Im(\gamma_1)}  \, d \PPP_1(\gamma_1)
= |z|_{\PPP_1} 
\end{split}
\end{equation}
The same argument shows that $|z|_\PPP \geq |z|_{\PPP_2}$. In order to prove the second inequality in \eqref{ts:density-conc} we employ the fact that $\max\{ 1_{z\in Im(\gamma_1)} , 1_{z\in Im(\gamma_2)}\}  \leq  1_{z\in Im(\gamma_1)} +1_{z\in Im(\gamma_2)} $ and the same computations as above. This concludes the proof of (1). 
\bigskip

In order to prove (2), we observe that for every $\gamma_1$ and $\gamma_2$ with $\gamma_1(\infty)=\gamma_2(0)$, we can write
\begin{equation}\label{sommabuona}
\begin{split}
A_{\phi}(conc(\gamma_1,\gamma_2))&= \int_{0}^{\infty} \phi(conc(\gamma_1,\gamma_2)(t))\Big| {\frac{d}{dt}conc(\gamma_1,\gamma_2)}( t)\Big| dt\\
&=\int_{0}^{T(\gamma_1)} \phi(\gamma_1( t))|\dot {\gamma_1}( t)| dt+ \int_{T(\gamma_1)}^\infty \phi(\gamma_2(t-T(\gamma_1)))|\dot {\gamma_2}(t-T(\gamma_1))| dt\\
&=\int_0^\infty \phi(\gamma_1( t)) |\dot\gamma_1(t) | \, dt + \int_0^\infty \phi(\gamma_2( t)) |\dot\gamma_2(t) | \, dt=A_{\phi}(\gamma_1)+A_{\phi}(\gamma_2).
\end{split}
\end{equation}
We plug the map $\varphi:=A_{\phi}$ in equation \eqref{eqn:defn-chi-P} to deduce
\begin{equation*}
\begin{split}
\int A_{\phi}&(\gamma) \, d \PPP(\gamma) = \int A_{\phi}(conc(\gamma_1, \gamma_2)) \, dP(\gamma_1,\gamma_2)\overset{\eqref{sommabuona}}{=} \int A_{\phi}(\gamma_1)+A_{\phi}( \gamma_2) \, dP(\gamma_1,\gamma_2)\\
&=\int A_{\phi}(\gamma_1)\, dP(\gamma_1,\gamma_2)+\int A_{\phi}( \gamma_2) \, dP(\gamma_1,\gamma_2)=\int A_{\phi}(\gamma_1)\, d\PPP_1(\gamma_1)+\int A_{\phi}( \gamma_2) \, d\PPP_2(\gamma_2),
\end{split}
\end{equation*}
which gives (2).

\bigskip

Integrating the second inequality in \eqref{ts:density-conc} with respect to $\Haus^1$ along the set $\{|z|_{\PPP_1}>0\} \cup \{|z|_{\PPP_2}>0\}$ (which is $1$-rectifiable as long as $\PPP_1$ and $\PPP_2$ have finite $\alpha$-mass), we deduce the second inequality in \eqref{ts:ts-conc}. 
For the first inequality, for every $\gamma_1$ and $\gamma_2$ with $\gamma_1(\infty)=\gamma_2(0)$, we use \eqref{sommabuona} with $\phi:=|\cdot|_{\PPP}^{\alpha-1}$ 
and by \eqref{ts:density-conc} we deduce
\begin{equation*}
\begin{split}
\int_0^\infty {|\conc(\gamma_1,\gamma_2)(t) |_\PPP^{\alpha-1}} \Big|\frac{d}{dt}\conc(\gamma_1,\gamma_2)(t) \Big| \, dt &{=}\int_0^\infty {|\gamma_1(t) |_\PPP^{\alpha-1}} |\dot\gamma_1(t) | \, dt
+
\int_0^\infty {|\gamma_2(t) |_\PPP^{\alpha-1}} |\dot\gamma_2(t) | \, dt
\\&{\leq}\int_0^\infty {|\gamma_1(t) |_{\PPP_1}^{\alpha-1}} |\dot\gamma_1(t) | \, dt
+
\int_0^\infty {|\gamma_2(t) |_{\PPP_2}^{\alpha-1}} |\dot\gamma_2(t) | \, dt
.
\end{split}
\end{equation*}
Integrating the previous inequality with respect to $P$, we get that 
\begin{equation*}
\begin{split}
\E(\PPP) &= \int\int_0^\infty {|\gamma(t) |_\PPP^{\alpha-1}} |\dot\gamma(t)| \, dt \, d \PPP(\gamma)= \int\int_0^\infty {|\conc(\gamma_1,\gamma_2) |_\PPP^{\alpha-1}} \Big|\frac{d}{dt}\conc(\gamma_1,\gamma_2) \Big| \, dt \, d P(\gamma_1,\gamma_2)\\
&
{\leq} \int \int_0^\infty {|\gamma_1(t) |_{\PPP_1}^{\alpha-1}} |\dot\gamma_1(t) | \, dt  \, d P(\gamma_1,\gamma_2) +
\int \int_0^\infty {|\gamma_2(t) |_{\PPP_2}^{\alpha-1}} |\dot\gamma_2(t) | \, dt \, d P(\gamma_1,\gamma_2) 
\\
&= \int \int_0^\infty {|\gamma_1(t) |_{\PPP_1}^{\alpha-1}} |\dot\gamma_1(t) | \, dt  \, d \PPP_1(\gamma_1) +
\int \int_0^\infty {|\gamma_2(t) |_{\PPP_2}^{\alpha-1}} |\dot\gamma_2(t) | \, dt \, d \PPP_2(\gamma_2) 
\\
&= \E(\PPP_1)+ \E(\PPP_2),
\end{split}
\end{equation*}
which yields (3).

\bigskip

To show (4), we consider the disintegration of $\PPP_1$ and $\PPP_2$ with respect to the common marginal $\nu$. Let $\{\PPP^x_1\}_{x\in X}$ and $\{\PPP^x_2\}_{x\in X}$ be families of probability measures representing the disintegration of $\PPP_1$ and $\PPP_2$ with respect to $\nu$. In other words, for every $x\in \R^d$, $\PPP^x_1 \in \PP(\Lip)$ is supported on curves that end at $x$ and $\PPP^x_2$ is supported on curves that begin at $x$, and
$$\PPP_1 = \nu \otimes \PPP_1^x, \qquad \PPP_2 = \nu \otimes \PPP_2^x.$$

We define
$$P:=\nu \otimes (\PPP_1^x \times \PPP_2^x),$$ 
namely for every $C^0$ test function $\varphi: \Lip \times \Lip \to \R$
\begin{equation}
\label{eqn:defn-chi-c}
\int \varphi(\gamma_1,\gamma_2) \, d P(\gamma_1,\gamma_2) = \int \varphi(\gamma_1, \gamma_2) \, d \PPP_1^x(\gamma_1)\, d \PPP_2^x(\gamma_2) \, d\nu(x).
\end{equation}
The measure $P$ satisfies all the properties in Definition~\ref{defn:set-of-conc}; indeed $(\rho_i)_\# P = \nu \otimes \PPP_i^x = \PPP_i$ and for $\nu$-a.e. $x$, for $\PPP_1^x \times \PPP_2^x$-a.e. $(\gamma_1,\gamma_2)$ we have that $\gamma_1(\infty) = x = \gamma_2(0)$. We define 
$$\qquad \PPP := \conc_\# P.$$
For every  test function $\psi:X\to\R$ we can compute
\begin{equation}
\begin{split}
\int \psi(x)\, d(e_0)_\# \PPP(x)&= \int \psi(e_0(\gamma))\, d \PPP(\gamma)= \int \psi(e_0(conc(\gamma_1,\gamma_2)))\, d P(\gamma_1,\gamma_2)\\
&\overset{\eqref{eqn:defn-chi-c}}{=} \int \psi(e_0(conc(\gamma_1,\gamma_2)))\, d \PPP_1^x(\gamma_1)\, d \PPP_2^x(\gamma_2) \, d\nu(x)\\
&= \int \psi(e_0(\gamma_1))\, d \PPP_1^x(\gamma_1) \, d\nu(x){=} \int \psi(e_0(\gamma_1))\, d \PPP_1(\gamma_1) \\
&= \int \psi(x)\,d(e_0)_\# \PPP_1(x)= \int \psi(x)\,d\mu^-(x)
\end{split}
\end{equation} 
and analogously
\begin{equation}
\begin{split}
\int \psi(x)\, d(e_\infty)_\# \PPP(x)&\overset{\eqref{eqn:defn-chi-c}}{=} \int \psi(e_\infty(conc(\gamma_1,\gamma_2)))\, d \PPP_1^x(\gamma_1)\, d \PPP_2^x(\gamma_2) \, d\nu(x)
\\ &{=} \int \psi(e_\infty(\gamma_2))\, d \PPP_2(\gamma_2) = \int \psi(x)\,d(e_\infty)_\# \PPP_2(x)= \int \psi(x)\,d\mu^+(x).
\end{split}
\end{equation}
 We deduce that $\PPP \in \TP(\mu^-,\mu^+)$ and we conclude that $\PPP$ is a concatenation of $\PPP_1$ and $\PPP_2$.

\bigskip

To prove (5), let $P, \tilde P \in \M(\Lip\times \Lip )$ be associated to $\PPP, \tilde \PPP$ as in Definition~\ref{defn:set-of-conc}. Then $P+ \tilde P$ verifies the properties of Definition~\ref{defn:set-of-conc} (where in property (2) we replace $\PPP_i$ with $\PPP_i+ \tilde \PPP_i$) and therefore $\PPP+\tilde{\PPP}=  \conc_\# (P+\tilde P) \in \TP(\mu^-+ \tilde\mu^-,\mu^+ + \tilde\mu^+)$ is a concatenation of $\PPP_1+\tilde\PPP_1$ and $\PPP_2+\tilde\PPP_2$.
\end{proof}


When dealing with the mailing problem, the concatenation of two traffic plans needs to take into account the coupling between the initial and the final measures. In general it is not true that, given two traffic plans $\PPP_1$ and $\PPP_2$ with a common marginal and given a coupling $\pi$ between the initial measure of $\PPP_1$ and the final measure of $\PPP_2$, there exists a concatenation of $\PPP_1$ and $\PPP_2$ with coupling $\pi$. On the other hand, in the special case when the common marginal is a Dirac delta, we are allowed to prescribe any coupling $\pi$, as shown in the next lemma.

\begin{lemma}[(Concatenation with prescribed coupling through a Dirac delta)]\label{lemma:conc}
Let $\PPP_1 \in \TP(\mu^-, \delta_{x_0})$, $\PPP_2 \in \TP(\delta_{x_0}, \mu^+)$, and let $\pi \in \PP(X \times X)$ with $(e_1)_\# \pi = \mu^-$ and $(e_2)_\# \pi = \mu^+$. Then there exists a concatenation $\PPP$  of $\PPP_1$ and $\PPP_2$ such that
$$\PPP \in \TP(\pi).$$
\end{lemma}
\begin{proof}
Let $\{ \PPP^x_1\}_{x\in X}$ be a $1$-parameter family of probability measures representing the disintegration of $\PPP_1$ with respect to the first marginal. In particular, for $\mu^-$-a.e. $x$, $\PPP^x_1$ is supported on curves that begin at $x$. Similarly, let $\{\PPP_2(x)\}_{x\in X}$ be a $1$-parameter family of probability measures representing the disintegration of $\PPP_2$ with respect to the second marginal
$$\PPP_1 = \mu^-(x) \otimes \PPP^x_1 \qquad \PPP_2 = \mu^+(x) \otimes \PPP^x_2.$$

We define the traffic plan $\PPP$ through the disintegration
$$P := \pi(x,y) \otimes(\PPP_1^x \times \PPP_2^y), \qquad \PPP:=  conc_\# P
,$$
recalling that this means that the measure $\PPP$ is the unique measure that satisfies for every bounded Borel test function $\varphi: \Lip \to \R$
\begin{equation}
\label{eqn:defn-chi-conc1}
\int \varphi(\gamma) \, d \PPP(\gamma) = \int \varphi(conc(\gamma_1, \gamma_2)) \, d \PPP_1^x(\gamma_1)\, d \PPP_2^y(\gamma_2) \, d\pi(x,y).
\end{equation}
We notice that $P$ satisfies the properties in Definition~\ref{defn:set-of-conc}.

We observe that $\PPP \in \TP(\pi)$. To verify this, we test \eqref{eqn:defn-chi-conc1} with $\varphi \circ (e_0,e_\infty):\Lip \to \R$ 
\begin{equation*}
\begin{split}
\int \varphi(x,y) \, d(e_0,e_\infty )_\# \PPP(x,y) &= \int \varphi(e_0(\gamma), e_\infty(\gamma)) \, d \PPP(\gamma) \\&= \int \varphi(\gamma_1(0), e_\infty(\gamma_2)) \, d\PPP_1^x(\gamma_1)\, d\ \PPP_2^y(\gamma_2) \, d\pi(x,y)
\\&= \int \varphi(x, y) \, d \PPP_1^x(\gamma_1)\, d\ \PPP_2^y(\gamma_2) \, d\pi(x,y)
=\int \varphi(x,y) \, d\pi(x,y).
\end{split}
\end{equation*}
This shows that $\PPP$ is a concatenation of $\PPP_1$ and $\PPP_2$ with $\PPP \in \TP(\pi)$.
\end{proof}
With the previous lemma we prove the equivalent of Theorem \ref{irrigability} for the mailing problem.
\begin{corollary}[(Upper bound on the energy for the mailing problem)]\label{cor:small-balls-en-con-assegn} Let $\alpha>1-\frac{1}{d}$. Then there exists constant $c(d,\alpha)$ such that if $\pi \in \M(X \times X)$ is supported on a set of diameter less than or equal than $L$, then 
	$$\E(\pi) \leq  c(d,\alpha) \pi(X\times X)^\alpha L.$$
\end{corollary}
\begin{proof}
	Let $x_0 \in X $ be any point in the support of $(e_1)_\#\pi$. By Theorem~\ref{irrigability} there exist two a traffic plans $\PPP_1\in \TP((e_1)_\#\pi, \delta_{x_0})$ and $\PPP_2\in \TP(\delta_{x_0}, (e_2)_\#\pi)$ satisfying 
	$$\E(\PPP_i)\leq C_{\alpha,d} [(e_1)_\#\pi](X)^\alpha L= C_{\alpha,d} \pi(X\times X)^\alpha L.$$
	Applying Lemma \ref{lemma:conc} to $\PPP_1$ and $\PPP_2$ and Lemma \ref{lemma:all-you-want-on-conc} (3), we find that 
	$$\E(\pi) \leq C_{\alpha,d} \big(\mu^-(X)^\alpha+\mu^+( X)^\alpha\big) L = 2 C_{\alpha,d}\pi(X\times X)^\alpha L.\qedhere$$
\end{proof}

\begin{remark}\label{semtpc}
A consequence of Corollary \ref{cor:small-balls-en-con-assegn} is that, above the critical threshold, there exists a constant $C=C(X,d, \alpha)$ such that if an optimal traffic plan $\PPP \in \OTP(\pi_{\PPP})$ of unit mass is supported on curves which are parametrized by arc-length, then $\PPP\in\TP_C$. Indeed we have
$$c(d,\alpha) \text{diam}(X) \geq \E(\PPP)=\int_{\Lip} \int_{\R^+} |\gamma( t)|_{\PPP}^{\alpha-1}|\dot \gamma(t)| dt\, d\PPP(\gamma)\geq \int_{\Lip} T(\gamma)\, d\PPP(\gamma).$$
In particular, the energy $\E$ is always lower semicontinuous along a converging sequence of such optimal traffic plans (compare with Lemma \ref{TPC}).
\end{remark}

We conclude this section on concatenations with a generalization of Lemma \ref{lemma:conc}. Roughly speaking, it says that, if we have three traffic plans $\PPP_1, \PPP_2, \PPP_3$, which are ``admissible'' for a concatenation, i.e. $(e_\infty)_\#\PPP_i= (e_0)_\#\PPP_{i+1}$, for $i=1,2$, and moreover ${(e_\infty)}_\sharp\PPP_2$ is a Dirac delta at a point $x_0$, then we can realize any coupling between the initial measure $(e_0)_\#\PPP_{1}$ and the final measure $(e_\infty)_\#\PPP_{3}$, via a suitable concatenation of the three traffic plans. Generalizing this construction, one could prove that whenever two or more traffic plan which can be concatenated have a Dirac delta in one of the intermediate marginals, then every coupling between the first and the last marginal can be realized by a concatenation.

\begin{lemma}\label{lemma:finiamola}
	Let $x_0\in \R^d$, $\mu_1,\mu_2, \mu_3$ be measures on $\R^d$ with the same mass. Let $\PPP_0$ be a traffic plan between $\mu_1$ and $\mu_3$, $\PPP_1$ between $\mu_1$ and $\mu_2$, $\PPP_2$ between $\mu_2$ and $\delta_{x_0}$, $\PPP_3$ between $\delta_{x_0}$ and $\mu_3$.
	
	Then there exists a concatenation $\PPP_4$ between $\PPP_2$ and $\PPP_3$ and a concatenation $\PPP_5$ between $\PPP_1$ and $\PPP_4$ such that $\PPP_5$ has the same coupling as $\PPP_0$.	
\end{lemma}
Only in the proof of this lemma, we will use the following notation. For $1\leq j\leq i$, we use 
$$e_j: \underbrace{\R^d \times\dots\times \R^d}_{\mbox{$i$ times}} \to \R^d$$
to denote the projection on the $j$-th copy of $\R^d$.

The proof will rely on the following lemma, known as gluing lemma, which can be found in \cite[Lemma 5.5]{santambrogiobook}. This time $e_1,e_2,e_3$ are the projections on $X,Y,Z$ respectively.
\begin{lemma}\label{gluing}
Let $X, Y, Z$ be locally compact, separable metric spaces. Let $\pi_{1}\in \PP(X \times Y)$ and $\pi_{2}\in \PP(Y \times Z)$. Then there exists $\pi \in  \PP(X \times Y\times Z)$ such that $(e_1,e_2)_\#\pi = \pi_1$ and $(e_2,e_3)_\#\pi = \pi_2$. 
\end{lemma}
\begin{proof}[Proof of Lemma~\ref{lemma:finiamola}] For every $i=0,\dots,3$ we set $\pi_i = (e_0, e_\infty)_\# \PPP_i$.
Applying Lemma \ref{gluing}, we build $ \tilde \pi \in \PP(\R^d_x \times \R^d_y \times \R^d_z)$ such that $(e_1,e_2)_\#\tilde \pi = \pi_1$ and $(e_2,e_3)_\#\tilde \pi = \pi_2$. Next, we apply again Lemma \ref{gluing} between $\tilde \pi$ (commuting the spaces in the product in the following order: $(\R^d_y \times \R^d_z) \times \R^d_x$) and $\pi_0$ to obtain, after commuting the spaces back, $\pi \in \PP(\R^d_x \times \R^d_y \times \R^d_z \times \R^d_w)$ such that
$$(e_1,e_2)_\# \pi =\pi_1, \qquad (e_1,e_4)_\# \pi =\pi_0.$$
We notice moreover that  
$$(e_2,e_3)_\# \pi =\pi_2, \qquad(e_3,e_4)_\# \pi =\pi_3,  $$
because $(e_3)_\#\pi = \delta_{x_0}$  and there exists a unique coupling $\pi_2$ with marginals $\mu_2$ and $\delta_{x_0}$ and a unique coupling $\pi_3$ with marginals $\delta_{x_0}$ and $\mu_3$.

Denote by $\PPP_4$ any concatenation of $\PPP_2$ and $\PPP_3$ which realizes the coupling $\pi_4:=(e_2,e_4)_\#\pi$ (it can be obtained by Lemma~\ref{lemma:conc}). 
To construct $\PPP_5$, we disintegrate $\PPP_1$ and $\PPP_4$ with respect to $(e_0,e_\infty)$ to get 
$$\PPP_1 = \pi_1(x,y) \otimes \PPP^{x,y}_1 = \pi(x,y,z,w) \otimes \PPP^{x,y}_1$$
$$\PPP_4 = \pi_4(x,y)\otimes \PPP^{y,w}_4 = \pi(x,y,z,w) \otimes  \PPP^{y,w}_4.$$
We can now define 
$$\PPP_5 := {\rm conc}_\#\big( \pi(x,y,z,w) \otimes  (\PPP^{x,y}_1 \times \PPP^{y,w}_4) \big).$$
Since for every $x,y,z,w$ the measures $\PPP_1^{x,y}$ and $\PPP^{y,w}_4$ are supported on two sets of curves with a common extreme in $y$, their concatenation is well defined. Moreover, we need to verify property (2) of Definition \eqref{defn:set-of-conc}
$$(\rho_1)_\# \big(  \pi(x,y,z,w)\otimes  (\PPP^{x,y}_1 \times \PPP^{y,w}_4)\big) = 
\pi(x,y,z,w) \otimes (\pi_1)_\#(\PPP^{x,y}_1 \times \PPP^{y,w}_4) 
=
 \pi(x,y,z,w) \otimes \PPP^{x,y}_1 = \PPP_1
$$
and similarly 
$(\rho_2)_\# \big(  \pi(x,y,z,w) \otimes (\PPP^{x,y}_2 \times \PPP^{y,w}_4)\big) = \PPP_2
$,
so that $\PPP_5$ is a well defined concatenation between $\PPP_1$ and $\PPP_4$. 
Finally we check that $\PPP_0$ and $\PPP_5$ have the same coupling
\begin{equation*}
\begin{split}
(e_0,e_\infty)_\# \PPP_5 &= {\big((e_0,e_\infty) \circ{\rm conc}\big)}_\#\big(  \pi(x,y,z,w) \otimes  (\PPP^{x,y}_1 \times \PPP^{y,w}_4) \big)
\\&= \pi(x,y,z,w) \otimes  \big( {\big((e_0,e_\infty) \circ{\rm conc}\big)}_\#(\PPP^{x,y}_1 \times \PPP^{y,w}_4)  \big)
\\& = \pi(x,y,z,w) \otimes \delta_{x,w} = (e_1,e_4)_\# \pi = \pi_0= (e_0,e_\infty)_\# \PPP_0.\qedhere
\end{split}
\end{equation*}
\end{proof}

%
%
%

\section{Stability for the mailing problem}

In this section we prove a slightly more refined version of the stability Theorem~\ref{thm:main}, which is more suited to the application presented in Section~\ref{s:stab}.
We deduce Theorem~\ref{thm:main} as a corollary.
\begin{theorem}\label{thm:good-version}
	Let $\alpha>1-\frac{1}{d}$. Let $\pi_\infty \in \PP(X\times X)$ with marginals $\mu^-$ and $\mu^+$ and $\{ \pi_n\}_{n\in\N}\subseteq \PP(X\times X)$ be a sequence with marginals $\mu^-_n$, $\mu^+_n$ such that
	\begin{equation}\label{hp:supp-n-convergence1}
\pi _n \weak \pi_\infty \qquad \mbox{weakly$^*$ in the sense of measures}.
	\end{equation}
	Then for every $\e>0$ there exists $n_0\in \N$ such that for every sequence $\{\PPP_i \in \TP(\pi_i)\}_{i\in\N\cup\{\infty\}}$ and for every $n,m\geq n_0$ (with $m= \infty$ possibly), there exist two traffic plans $\PPP^1_{n,m} \in \TP(\mu^-_n, \mu^-_m)$, $\PPP^2_{n,m}\in \TP(\mu^+_m, \mu^+_n)$ such that
	$$ \E (\PPP^1_{n,m}) \leq \e, \qquad \E (\PPP^2_{n,m}) \leq \e$$
and there exists a traffic plan obtained as a concatenation between $\PPP_m$ and $\PPP^2_{n,m}$, and a further concatenation $\tilde\PPP_{n,m}$ between $\PPP^1_{n,m}$ and the latter concatenation such that $\tilde\PPP_{n,m}\in\TP(\pi_n)$.

\end{theorem}
\begin{proof}[Proof of Theorem~\ref{thm:main}] We apply Theorem~\ref{thm:good-version} with $m= \infty$, being $\PPP_\infty\in\OTP(\pi_\infty)$, to find a competitor $\tilde \PPP_{n,\infty} \in \TP(\pi_n)$ given by a concatenation of $\PPP_\infty$ and  $\PPP^2_{n,\infty}$ and a further concatenation of $\PPP^1_{n,\infty}$ and the latter concatenation; hence with the energy estimate
$$\E(\tilde \PPP_{n,\infty}) \leq \E(\PPP^1_{n,\infty})+ \E(\PPP_\infty) + \E(\PPP_{n,\infty}^2).$$
Moreover, for every $\e>0$, for $n$ large enough, the ``connections'' have small energy $\E(\PPP^1_{n,\infty}) , \E(\PPP^2_{n,\infty}) \leq \e$.
Letting $n \to \infty$, by the lower semicontinuity of the energy (see Lemma \ref{TPC}) and since $\PPP_n\in\OTP(\pi_n)$, we have
$$\E(\PPP) \leq \liminf_{n\to \infty}\E( \PPP_n) \leq \liminf_{n\to \infty}\E(\tilde \PPP_{n,\infty}) \leq \E(\PPP_\infty) + 2\e.
$$
Since the previous formula holds for every $\e>0$, we deduce that $\PPP\in\OTP(\pi)$.

To conclude the proof of Theorem~\ref{thm:main}, we need to show that $\E( \PPP_n)\to \E( \PPP)$. To this aim, we apply Theorem~\ref{thm:good-version} again with $m= \infty$, but $\PPP_\infty$ will be now chosen equal to $\PPP$ (that we already know is an optimal traffic plan). We deduce that for every $\e>0$ and $n$ large enough there exists a competitor $\tilde \PPP_{n,\infty} \in \TP(\pi_n)$ given by a concatenation of $\PPP$ and $\PPP^2_{n,\infty}$ and a further concatenation between $\PPP^1_{n,\infty}$ and the latter concatenation; hence with the energy estimate
$$\E(\tilde \PPP_{n,\infty}) \leq \E(\PPP^1_{n,\infty})+ \E(\PPP) + \E(\PPP_{n,\infty}^2).$$
Moreover, the connections have small energy $\E(\PPP^1_{n,\infty}) , \E(\PPP^2_{n,\infty}) \leq \e$.
Letting $n \to \infty$, as above we have
$$\E(\PPP) \leq \liminf_{n\to \infty}\E( \PPP_n) \leq\limsup_{n\to \infty}\E( \PPP_n) \leq \limsup_{n\to \infty}\E(\tilde \PPP_{n,\infty}) \leq \E(\PPP) + 2\e.
$$
Since the previous inequality holds for every $\e>0$, we deduce that $\E( \PPP_n)\to \E( \PPP)$. 
\end{proof}

The proof of theorem \ref{thm:good-version} relies on the following lemma, which employs in an essential way the condition $\alpha>1-\frac{1}{d}$. It builds a suitable covering of the support of a given measure, with balls satisfying the smallness condition in \eqref{eqn:cov3} below.

\begin{lemma}\label{lemma:meas-cov}
Let $\alpha>1-\frac{1}{d}$, $\e>0$ and let $\mu \in \PP(X)$. Then there exists a countable family of balls $\{B_i = B(x_i,r_i)\}_{i\in \N}$ such that
\begin{equation}
\label{eqn:cov1}
\mu\Big(\R^d \setminus \bigcup_{i=1}^\infty B_i \Big) = 0,
\end{equation}
\begin{equation}
\label{eqn:cov2}
\mu(\partial B_i) = 0 \qquad \mbox{for every }i\in \N,
\end{equation}
\begin{equation}
\label{eqn:cov3}
\sum_{i=1}^\infty r_i \mu(B_i)^\alpha <\e.
\end{equation}
\end{lemma}
\begin{remark}When $\mu = f \Leb^d \trace X$ for some $f\in L^1 \cap L^\infty(X)$, the previous lemma is trivial. Indeed, for any $r_0>0$, we can find a countable covering of $X$ made by balls $\{B(x_i,r_i)\}_{i \in \N}$ with $r_i\leq r_0$ and
$\sum_{i=1}^\infty r_i^{d} \leq \Leb^d(X)+1$. We estimate
\begin{equation*}
\begin{split}
\sum_{i=1}^\infty r_i \mu(B_i)^\alpha  &\leq 
\sum_{i=1}^\infty r_i^{1+\alpha d} \|f \|_{L^\infty}^\alpha
\leq 
r_0^{1+\alpha d - d}\sum_{i=1}^\infty r_i^{d} \|f \|_{L^\infty(\R^d)}^\alpha
\leq 
r_0^{1+\alpha d - d} ( \Leb^d(X)+1) \|f \|_{L^\infty}^\alpha.
\end{split}
\end{equation*}
By a suitable choice of $r_0$ and since $1+\alpha d - d>0$, this quantity can be made arbitrarily small.
%
%
\end{remark}
\begin{proof}[Proof of Lemma~\ref{lemma:meas-cov}]
Let us consider 
$$\delta:= \frac 1 {1-\alpha}-d$$
and notice that $\delta >0$ since $\alpha>1-\frac{1}{d}$.
Let $D \subseteq X$ be the set of points such that the Radon-Nikodym density of $\mu$ with respect to $\Leb^d$ exists and is strictly positive (and possibly infinite)
$$D:=\Big \{ x\in X: \mbox{ there exists } \frac{d\mu}{d \Leb^d}(x):=\lim_{r \to \infty} \frac{\mu(B_r(x))}{\omega_d r^d} \in (0,\infty]
\Big\}.$$
By the Besicovich derivation theorem (see \cite[Theorem 2.22]{AFP}) this set supports the measure $\mu$; we aim to find a covering of $\mu$-a.e. point of $D$ made by balls that satisfy \eqref{eqn:cov2} and \eqref{eqn:cov3}.

For any point $x\in D$, we claim that for every $r$ small enough (depending on $x$)
\begin{equation}
\label{eqn:exit-time}
\mu(B(x,r))^\alpha r \leq {\e} \mu(B(x,r)).
\end{equation}
Indeed, this condition \eqref{eqn:exit-time} can be rewritten as
\begin{equation}
\label{eqn:exit-time-equiv}
\Big(\frac{1}{\e}\Big)^{d+\delta} r^\delta \leq \frac{\mu(B(x,r))}{r^d}.
\end{equation}
Since the point $x$ has positive Radon-Nikodym density $\frac{d\mu}{d \Leb^d}(x)$, the right-hand side in \eqref{eqn:exit-time-equiv} is greater than $\frac{\omega_d}{2}\frac{d\mu}{d \Leb^d}(x)>0$ for $r$ small enough. On the other hand, the left-hand side of \eqref{eqn:exit-time-equiv} converges to $0$ as $r \to 0$. Hence the claim \eqref{eqn:exit-time} holds.

For every $x\in D$, we also observe that for every $r$ apart at most countably many
\begin{equation}
\label{eqn:no-meas-bound}
\mu(\partial B(x,r)) = 0.
\end{equation}

The set $D$, because of the previous observation, has a fine covering made by all balls $B(x,r)$ centered at any point in $D$ and satisfying  \eqref{eqn:exit-time} and \eqref{eqn:no-meas-bound}.
Hence, by the Vitali-Besicovich covering lemma (see 
\cite[Theorem 2.19]{AFP}), we can find a countable family of pairwise disjoint balls $\{B_i = B(x_i,r_i)\}_{i\in \N}$ satisfying the properties \eqref{eqn:exit-time} and \eqref{eqn:no-meas-bound} that cover $D$. 
Hence, we estimate the left-hand side of \eqref{eqn:cov3} thanks to \eqref{eqn:exit-time} 
\begin{equation*}
\begin{split}
\sum_{i=1}^\infty r_i \mu(B_i)^\alpha 
&\leq \e \sum_{i=1}^\infty \mu(B_i) = \e.\qedhere
\end{split}
\end{equation*}
\end{proof}


\begin{proof}[Proof of Theorem~\ref{thm:good-version}]We develop the proof in several steps, briefly described here.

In the following, we consider two countable families of balls $B_i^\pm$ as in Lemma~\ref{lemma:meas-cov}.
We select two finite subfamilies, indexed by $1,..., N^\pm$, which cover a big portion of the measures $\mu^\pm$ (Step 1 and 2).

In Step 3, we decompose the traffic plans $\PPP_n$, $\PPP_m$ in a sum of smaller traffic plans  $\PPP_n^{i,j}$ and  $\PPP_m^{i,j}$, according to the fact that each curve begins in a certain $B_i^- \setminus \big(\cup_{k=1}^{i-1}B_k^- \big)$, for some $i=1,...,N^-$ and ends in $B_j^+ \setminus \big(\cup_{k=1}^{j-1}B_k^+ \big)$ for some $j=1,...,N^+$.

We define the residuum as the part of the traffic plans $\PPP_n$ and $\PPP_m$ that either begins outside $\cup_{k=1}^{N^-}B_k^-$ or ends outside $\cup_{k=1}^{N^+}B_k^+$; the construction which takes care of the residuum is carried out in Step 7 and is based on the smallness of the total mass carried by it.

In Step 4, we connect the initial points of each $\PPP_m^{i,j}$ with the initial points of $\PPP_n^{i,j}$ by means of a traffic plan; thanks to the choice of the balls given by Lemma~\ref{lemma:meas-cov}, the sum of these traffic plans as $i$ and $j$ vary (denoted in the following by $\PPP^I$), can be built with arbitrarily small cost.

A similar construction works for the final points (Step 5) and we will denote the corresponding plan as $\PPP^{IV}$; in addition, doing this construction more carefully, one can ensure that the traffic plan connecting the final points of each $\PPP_n^{ij}$ with the final points of $\PPP_m^{ij}$ passes through the center of the ball ${B_j^+}$.

This last observation and Lemma~\ref{lemma:finiamola} ensure that we can find a concatenation of the traffic plans $\PPP^I$, $\PPP_m- \PPP_m^{res}$ and $\PPP^{IV}$ which has the same marginals of $\PPP_n-\PPP_n^{res}$, as shown in Step 6.

%

{\it Step 1: construction of a covering of $\supp(\mu^\pm)$}. Let $\bar\e>0$ to be chosen later (in terms only of $d,\alpha, {\rm diam}(X),\e$).
Applying Lemma~\ref{lemma:meas-cov}, we get that there exists a finite number of balls $\{B_i^\pm = B(x^\pm_i,r^\pm_i)\}_{i\in \N}$ covering $\supp (\mu^\pm) $ and satisfying the following properties:
\begin{equation}\label{eqn:cov-limit-1}
\mu^\pm\Big(\R^d \setminus \bigcup_{i=1}^\infty B_i^\pm \Big) = 0,
\end{equation}
\begin{equation}\label{eqn:cov-limit-2}
\sum_{i=1}^\infty r_i^\pm \mu(B^\pm_i)^\alpha <\bar\e^\alpha,
\end{equation}
\begin{equation}\label{eqn:cov-limit-3}
\mu^\pm(\partial B_i^\pm) =0 \qquad \mbox{for every }i\in \N.
\end{equation}

\smallskip
{\it Step 2: choice of $N^\pm$ and of $n_0$}. We choose $N^\pm$ satisfying
$$\mu^\pm \Big( \bigcup_{i=1}^{N^\pm} {B_i^\pm}\Big) > 1-\bar\e.
$$
 For every $i\in \{ 1,..., N^\pm\}$ we define 
$$ C_{i}^\pm:=B_i^\pm \setminus \Big(\cup_{k=1}^{i-1}B_k^\pm \Big).$$ 

Without loss of generality, the coverings can be assumed not redundant, that is
\begin{equation}\label{nonzero}
\mu^-(C_i^-)>0, \qquad \mu^+(C_i^+)>0 \qquad \mbox{for every $i\in \N$}.
\end{equation}
We can fix $n_0$ large enough so that the following properties hold: 
\begin{equation}\label{eqn:choice-n-3}
\mu_n^\pm(C_i^\pm )\leq 2 \mu^\pm(C_i^\pm ) \qquad \mbox{$\forall n \geq n_0$ and $i \in \{1,..., N^\pm\}$,} 
\end{equation}
\begin{equation}\label{eqn:choice-n-2}
\mu^\pm_n \Big( \Big( \bigcup_{i=1}^{N^\pm} {B_i^\pm}\Big)^c\Big) \leq 2\bar\e,
\end{equation}
\begin{equation}\label{eqn:choice-n-1}
\pi_n(C_i^-\times C_j^+) - \pi_m(C_i^- \times C_j^+) \leq \frac{\bar\e}{N^- N^+} \quad \mbox{ $ \forall m,n \geq n_0$, $ (i,j) \in \{1,..., N^-\}\times \{1,..., N^+\}$}.
\end{equation}
Indeed, \eqref{eqn:choice-n-1} 
follows from assumption \eqref{hp:supp-n-convergence1} and \eqref{eqn:cov-limit-3}. Moreover, since \eqref{hp:supp-n-convergence1} implies that $\mu_n^\pm$ weakly converges to $\mu^\pm$, then \eqref{eqn:choice-n-3} follows from \eqref{nonzero} and \eqref{eqn:choice-n-2} follows from \eqref{eqn:cov-limit-1}.

%

\smallskip
{\it Step 3: decomposition of the traffic plans according to initial and final points of each curve.}
For every $n\in\N$ we define $\PPP^{i,j}_n$ to be the part of traffic plan $\PPP_n$ made by curves that start in $C_i^-$ and end in $C_j^+$, namely
$$\PPP^{i,j}_n := \PPP_n  \trace \{ \gamma: \gamma(0) \in C_i^-, \, \gamma(\infty)  \in C_j^+ \},$$
observing that the coupling induced by $\PPP^{i,j}_n$ is
$$(e_0, e_\infty)_\# \PPP^{i,j}_n = \pi_n \trace (C_i^-\times C_j^+).$$
For any  $i \in \{1,..., N^-\}$ and $j\in \{1,..., N^+\}$, we consider 
$$\mu^{i,j}_{n,\pm} := \mu^\pm_{\PPP^{i,j}_n},$$ 
$$
\alpha_{m,n}^{i,j}:= \min\Big\{1, \frac{\PPP_m^{i,j}(\Lip)}{\PPP_n^{i,j}(\Lip)} \Big\},
\qquad
\alpha_{n,m}^{i,j}:= \min\Big\{1, \frac{\PPP_n^{i,j}(\Lip)}{\PPP_m^{i,j}(\Lip)} \Big\}
$$
(with the agreement that the fractions are infinite if the denominator is $0$).
Let us consider rescalings of $ \PPP^{i,j}_n $ and $ \PPP^{i,j}_m $, namely
$$\tilde \PPP^{i,j}_n := \alpha_{m,n}^{i,j}\PPP^{i,j}_n, \qquad \tilde \PPP^{i,j}_m := \alpha_{n,m}^{i,j}\PPP^{i,j}_m
$$
and we consider their marginals
$$\tilde \mu^{\pm,i,j}_{n} := \mu^\pm_{\tilde \PPP^{i,j}_n}\quad\mbox{and}\quad \tilde \mu^{\pm,i,j}_{m} := \mu^\pm_{\tilde \PPP^{i,j}_m}.$$
With this definition, for every  $(i,j) \in \{1,..., N^-\}\times \{1,..., N^+\}$ we have $ \tilde \PPP^{i,j}_n \leq  \PPP^{i,j}_n$, $ \tilde \PPP^{i,j}_m \leq  \PPP^{i,j}_m$, 
$$\tilde \PPP^{i,j}_n (\Lip) =  \tilde \PPP^{i,j}_m (\Lip), \qquad \tilde \mu_n^{-,i,j}(\R^d) = \tilde \mu^{-,i,j}_m(\R^d)=\tilde\mu_n^{+,i,j}(\R^d) = \tilde \mu^{+,i,j}_m(\R^d).$$
Moreover, it holds that $\tilde \PPP_n \leq \PPP_n$, defining 
$$\tilde \PPP_n:=\sum_{i=1}^{N^-}\sum_{j=1}^{N^+}\tilde \PPP^{i,j}_n.$$
Next we introduce the residuum, namely $\PPP_n -  \tilde \PPP_n$, and we split it into two pieces as follows
$$\PPP_n  = \tilde \PPP_n+ \PPP^{res}_n = \tilde \PPP_n+ \PPP^{1,res}_n+ \PPP^{2,res}_n,$$
where
$$\PPP^{1,res}_n= \sum_{i=1}^{N^-}\sum_{j=1}^{N^+} (1-\alpha_{m,n}^{i,j}) \PPP_n^{i,j}, \qquad \PPP^{2,res}_n= \PPP_n-\sum_{i=1}^{N^-}\sum_{j=1}^{N^+} \PPP_n^{i,j}.$$
We show that $\PPP^{res}_n$ has small mass, namely
\begin{equation}
\label{eqn:piccolo-res}
\PPP^{res}_n(\Lip) \leq 5\bar\e.
\end{equation}
Indeed, since $\PPP_n^{i,j}(\Lip)=\PPP_n  ( \{ \gamma: \gamma(0) \in C_i^-, \, e_\infty(\gamma)  \in C_j^+ \} )= \pi_n(C_i^-\times C_j^+)$ and similarly for $\PPP_m^{i,j}(\Lip)$, and by  \eqref{eqn:choice-n-1}, we have
\begin{equation}
\begin{split}
\PPP^{1,res}_n (\Lip) &=\sum_{i=1}^{N^-}\sum_{j=1}^{N^+}  (1-\alpha_{m,n}^{i,j}) \PPP_n^{i,j}(\Lip)
\leq\sum_{i=1}^{N^-}\sum_{j=1}^{N^+}  \max\Big\{ 0, \PPP_n^{i,j}(\Lip)-  \PPP_m^{i,j}(\Lip) \Big\} 
\\&\leq 
\sum_{i=1}^{N^-}\sum_{j=1}^{N^+} |\pi_n(C_i^-\times C_j^+) - \pi_m(C_i^- \times C_j^+)|
 \leq \sum_{i=1}^{N^-}\sum_{j=1}^{N^+} \frac{\bar\e}{N^- N^+} = \bar\e.
 \end{split}
\end{equation}
As regards $\PPP^{2,res}_n$, it is supported on curves which either begin outside any $C_i^-$, $i=1,..., N^-$, or end outside any $C_j^+$, $j=1,..., N^+$, and therefore has small mass thanks to \eqref{eqn:choice-n-2}
\begin{equation}
\begin{split}
\PPP^{2,res}_n(\Lip) &= 1- \PPP_n \big(
 \{ \gamma: \gamma(0) \in \bigcup_{i=1}^{N^-} C_i^-, \, e_\infty(\gamma)  \in \bigcup_{j=1}^{N^+}C_j^+ \} \big)
 \\&\leq \PPP_n  \Big(
  \Big\{ \gamma: \gamma(0) \notin \bigcup_{i=1}^{N^-} C_i^- \Big\} \Big) +
 \PPP_n \Big(
 \Big\{ \gamma: e_\infty(\gamma)  \notin \bigcup_{j=1}^{N^+}C_j^+ \Big\} \Big)
 \\&\leq \mu_n^- \Big(\Big(
   \bigcup_{i=1}^{N^-} C_i^- \Big)^c \Big) +\mu_n^+ \Big(\Big(
   \bigcup_{j=1}^{N^+} C_j^+ \Big)^c \Big) 
 \leq 4 \bar\e.
\end{split}
\end{equation}

{\it Step 4: initial connection.}
Let us apply the Corollary~\ref{cor:small-balls-en-con-assegn} to the coupling
$$\pi^{I,i}=  \sum_{j=1}^{N^+} \tilde \mu_n^{-,i,j} \times \tilde \mu_m^{-,i,j}.
$$
We find a traffic plan $\PPP^{I,i} \in \TP (\pi^{I,i})$ 
with energy controlled by
$$\E(\PPP^{I,i}) \leq c(d,\alpha) \pi^{I,i}(\R^d\times \R^d)^\alpha r_i \leq c(d,\alpha)  \Big(\sum_{j=1}^{N^+} \tilde\mu_n^{-,i,j}(\R^d)\Big)^\alpha r_i \leq c(d,\alpha) \sum_{j=1}^{N^+} \mu_n^{-,i,j}(C_i^-)^\alpha r_i.
$$
Finally, recalling that $ \sum_{j=1}^{N^+} \mu_n^{-,i,j} \leq \mu_n$ and  \eqref{eqn:choice-n-3}, we obtain
$$\E(\PPP^{I,i}) \leq c(d,\alpha)\mu_n^{-}(C_i^-)^\alpha r_i  \leq 2 c(d,\alpha)\mu^{-}(C_i^-)^\alpha r_i.
$$
By means of Lemma~\ref{lemma:somma-plan}, we can write $\PPP^{I,i}$ as the following sum
$$\PPP^{I,i}:= \sum_{j=1}^{N^+} \PPP^{I,i,j} \qquad \mbox{where }\PPP^{I,i,j} \in \TP(\tilde \mu_n^{-,i,j} \times \tilde \mu^{-,i,j}_m).$$
Setting 
$$\PPP^{I}:= \sum_{i=1}^{N^-} \PPP^{I,i}$$
we obtain that its energy is small thanks to \eqref{eqn:cov-limit-2}
\begin{equation}
\label{eqn:enP1}
\E(\PPP^{I}) \leq \sum_{i=1}^{N^-} \E(\PPP^{I,i}) \leq 2c(d,\alpha)\sum_{i=1}^{N^-} \mu^{-}(C_i^-)^\alpha r_i \leq 2c(d,\alpha) \bar\e^\alpha.
\end{equation}

{\it Step 5: final connection: some useful traffic plans.}
Let us finally consider
$\PPP^{II, j} \in \TP \big(\sum_{i=1}^{N^-} \tilde \mu_m^{+,i,j}, |\sum_{i=1}^{N^-} \tilde \mu_m^{+,i,j}|\delta_{x_j} \big)$ (where we recall that $| \mu|$ denotes the total mass of any nonnegative measure $\mu$)
with
$$\E(\PPP^{II, j} ) \leq  c(d,\alpha)  \Big(\sum_{i=1}^{N^-} \tilde\mu_m^{+,i,j}(\R^d) \Big)^\alpha r_j 
\leq c(d,\alpha)\mu_m^{+}(C_j^+)^\alpha r_j  \leq 2 c(d,\alpha)\mu^{+}(C_j^+)^\alpha r_j.
$$
Setting 
$$\PPP^{II} := \sum_{j=1}^{N^+} \PPP^{II,j}$$
we obtain that its energy is small thanks to \eqref{eqn:cov-limit-2}
\begin{equation}\label{eqn:eP2}
\E(\PPP^{II}) \leq \sum_{j=1}^{N^+} \E(\PPP^{II,j}) \leq 2c(d,\alpha)\sum_{j=1}^{N^+} \mu^{+}(C_j^+)^\alpha r_j \leq 2c(d,\alpha) \bar \e^\alpha.
\end{equation}
By Lemma~\ref{lemma:somma-plan}, we can write
$$\PPP^{II,j}:= \sum_{i=1}^{N^-} \PPP^{II,i,j} \qquad \mbox{where }\PPP^{II,i,j} \in \TP\big( \tilde \mu_m^{+,i,j} 
\times  (|\tilde\mu_m^{+,i,j}|\delta_{x_j}) \big)$$
In a completely analogous way, we define
$$\PPP^{III}=\sum_{i=1}^{N^-}\sum_{j=1}^{N^+} \PPP^{III,i,j} \qquad \mbox{where }\PPP^{III,i,j} \in \TP\big( 
  (|\tilde \mu_n^{+,i,j}|\delta_{x_j})\times \tilde \mu^{+,i,j}_n\big)$$
with small energy
\begin{equation}\label{eqn:eP3}
\E(\PPP^{III})  \leq 2c(d,\alpha) \bar \e ^\alpha.
\end{equation}

{\it Step 6: Concatenation with correct coupling for fixed $i,j$}. For any $i \in \{1,..., N^-\}$ and $j\in \{1,..., N^+\}$, 
we apply Lemma~\ref{lemma:finiamola} with $\PPP_0 =  \tilde \PPP_n^{i,j}$, $\PPP_1$ a concatenation between $\PPP^{I,i,j}$ and $\tilde \PPP^{i,j}_m$, $\PPP_2= \PPP^{II,i,j}$, $\PPP_3 = \PPP^{III,i,j}$
to find a concatenation $\PPP^{IV,i,j}$ between $ \PPP^{II,i,j}$ and $\PPP^{III,i,j}$ and a further concatenation $ \PPP^{V,i,j}$ between $\PPP_1$ and $\PPP^{IV,i,j}$ with the same coupling as $\tilde \PPP_n^{i,j}$ 
\begin{equation}
\label{eqn:marginalsP-Vij}
(e_0,e_\infty)_\# \PPP^{V,i,j} = (e_0,e_{\infty})_\# \tilde \PPP_n^{i,j}.
\end{equation}
We next set
$$\PPP^{IV}:= \sum_{i=1}^{N^-} \sum_{j=1}^{N^+} \PPP^{IV,i,j}, \qquad \PPP^{V}:= \sum_{i=1}^{N^-} \sum_{j=1}^{N^+} \PPP^{V,i,j}$$
and we notice that by Lemma~\ref{lemma:all-you-want-on-conc}(5) $\PPP^{IV}$ is a concatenation between $\PPP^{II}$ and $\PPP^{III}$. Hence, by Lemma~\ref{lemma:all-you-want-on-conc}(3) together with the smallness of the energies of $\PPP^{II}$ and $\PPP^{III}$ we have
\begin{equation}\label{eqn:eP4}
\E(\PPP^{IV})\leq \E(\PPP^{II})+\E(\PPP^{III}) \leq 4c(d,\alpha) \varepsilon_1.
\end{equation}
Thanks to the computation of the marginals of $ \PPP^{V,i,j}$ in \eqref{eqn:marginalsP-Vij}, we deduce that
\begin{equation}
\label{eqn:marginalsP-V}
(e_0,e_\infty)_\# \PPP^{V} = (e_0,e_{\infty})_\# \tilde \PPP_n
\end{equation}

\bigskip

{\it Step 7: residuals.} Let $\PPP^{I,res}$ be an optimal traffic plan with coupling $(e_0)_\# \PPP_n^{res} \times (e_0)_\# \PPP_m^{res}$. By Corollary~\ref{cor:small-balls-en-con-assegn} and \eqref{eqn:piccolo-res} we have the energy bound
\begin{equation}
\label{en:Pres-1}
\E(\PPP^{I,res}) \leq  c(d,\alpha,{\rm diam}(X)) \PPP_n^{res}(\Lip)^\alpha \leq  5c(d,\alpha,{\rm diam}(X)) \bar\e^\alpha.
\end{equation}
Let $\PPP^{II,res}$ and $\PPP^{III, res}$ be optimal traffic plans with coupling $(e_\infty)_\# \PPP_m^{res} \times \delta_0$ and $\delta_0 \times (e_\infty)_\# \PPP_n^{res}$ respectively. By  Corollary~\ref{cor:small-balls-en-con-assegn} we have the energy bound
\begin{equation}\label{eqn:eP4res}
\E(\PPP^{II,res}) +\E(\PPP^{III,res})\leq  2c(d,\alpha,{\rm diam}(X)) \PPP_n^{res}(\Lip)^\alpha \leq  10c(d,\alpha,{\rm diam}(X)) \bar\e^\alpha.
\end{equation}
We apply Lemma~\ref{lemma:finiamola} with $\PPP_0 =  \PPP_n^{res}$, $\PPP_1$ a concatenation between $\PPP^{I,res}$ and $\PPP_m^{res}$, $\PPP_2= \PPP^{II,res}$, $\PPP_3 = \PPP^{III,res}$
to find a concatenation $\PPP^{IV,res}$ between $\PPP^{II,res}$ and $\PPP^{III,res}$ and a further concatenation $\PPP^{V,res}$ between $\PPP_1$ and $\PPP^{IV,res}$ with the same coupling as $\PPP^{res}_n$
\begin{equation}
\label{eqn:p5res-marg}
(e_0,e_\infty)_\# \PPP^{V,res}= (e_0,e_{\infty})_\# \PPP^{res}_n.
\end{equation}

{\it Step 8: Conclusion}. We set
$$\PPP^1_{n,m}: = \PPP^{I,res}+ \PPP^{I}, \qquad \PPP^2_{n,m} := \PPP^{IV,res}+ \PPP^{IV}, \qquad \PPP^5_{n,m} := \PPP^{V,res}+ \PPP^{V}$$
and we claim that they satisfy the properties in the statement of the theorem.

Indeed, the energy of $\PPP^1_{n,m}$ is estimated using \eqref{en:Pres-1} and \eqref{eqn:enP1} by
$$\E(\PPP^1_{n,m})\leq \E(\PPP^{I,res})+ \E(\PPP^{I})\leq 7c(d,\alpha,{\rm diam}(X)) \bar\e^\alpha\leq \e,$$
and the energy of $\PPP^2_{n,m}$ is estimated using \eqref{eqn:eP2}, \eqref{eqn:eP3}, and \eqref{eqn:eP4res} by
$$\E(\PPP^2_{n,m})\leq \E(\PPP^{IV,res})+ \E(\PPP^{IV})\leq 14c(d,\alpha,{\rm diam}(X)) \bar\e^\alpha\leq \e$$
where the last inequality in both lines follows by choosing $\bar\e$ suitably small in terms of $\e$.

By definition, $\PPP^V$ is a concatenation of $\PPP^I$ with a concatenation between $\tilde \PPP_m$ and $\PPP^{IV}$. Hence, adding term by term and employing Lemma~\ref{lemma:all-you-want-on-conc}(5), $\PPP^5_{n,m}$ is a concatenation of $\PPP^1_{n,m} = \PPP^{I,res}+ \PPP^{I}$ with a concatenation between $\PPP_m =\PPP^{res}_m +\tilde \PPP_m$ and $\PPP^2_{n,m} = \PPP^{IV,res}+ \PPP^{IV}$. This implies in particular that $\PPP^1_{n,m}$ has the same right-marginal as the initial marginal of $\PPP_m$, that is 
$\PPP^1_{n,m} \in \TP(\mu^-_n, \mu^-_m)$
Analogously, $\PPP^2_{n,m}\in \TP(\mu^+_n, \mu^+_m)$.

 Finally, the coupling identified by $\PPP^5_{n,m}$ is the same as the coupling of $\PPP_n$: indeed, we have already computed in \eqref{eqn:p5res-marg} and \eqref{eqn:marginalsP-V} the couplings of $\PPP^{V,res}$ and $\PPP^{V}$.
$$(e_0,e_\infty)_\# \PPP^{V} = (e_0,e_{\infty})_\# \PPP^{V,res}+ (e_0,e_\infty)_\# \PPP^{V} = (e_0,e_{\infty})_\# \PPP^{res}_n+ (e_0,e_{\infty})_\# \tilde \PPP_n = (e_0,e_{\infty})_\# \PPP_n.$$



\end{proof}

\section{Finiteness of the number of connected components}\label{s:stab}
We recall the definition of connected components of a traffic plan. This requires some preliminary notation. 
Given $t_1\leq t_2\in[0,\infty)$ we define the map $res_{t_1,t_2}:\Lip\to\Lip$ as
\begin{equation*}
[res_{t_1,t_2}(\gamma)](t):=
\begin{cases}
\gamma(t_1) \qquad &\mbox{if } t\leq t_1,
\\
\gamma(t) \qquad &\mbox{if }  t_1<t< t_2,
\\
\gamma(t_2) \qquad &\mbox{if } t\geq t_2.
\end{cases}
\end{equation*}

We say that a curve $\gamma\in\Lip$ is \emph{simple} if there exists no triple of times $t_1<t_2<t_3$ such that $\gamma(t_1)=\gamma(t_3)\neq\gamma(t_2)$. For every $x,y\in X$,
\begin{enumerate}
\item{} we consider the set $A_{x,y} \subset \Lip$ 
$$A_{x,y}:=\{\gamma\in\Lip: \gamma \mbox{ is simple and }\exists t_1, t_2 \in \R^+ : t_1\leq t_2, \, \gamma(t_1)=x, \, \gamma(t_2)=y\};$$
\item we denote $t_p$ the first time for which $\gamma(t_p)=p$ and we observe that ${res_{t_x,t_y}}$ is well defined on $A_{x,y}$; moreover, fixing an $x_0 \in X$, we can extend ${res_{t_x,t_y}}$ to be the constant curve in $x_0$;
\item{} we define the traffic plan $\PPP_{x,y}:=({res_{t_x,t_y}})_{\#}(\PPP\trace A_{x,y})$;
\item{} we denote 
$$\Gamma_{x,y}:=\{p:|p|_{\PPP_{x,y}}>0\}\cup\{p:|p|_{\PPP_{y,x}}>0\}.$$
\end{enumerate}

We say that a traffic plan $\PPP$ is \emph{loop free} if $\PPP$ is supported on simple curves. Optimal traffic plans for the mailing problem are loop free (see Proposition 4.9 \cite{BCM}), and we recall that they enjoy the single path property (see \cite[Proposition 7.4]{BCM}). Hence if $\PPP\in\OTP(\pi_\PPP)$, then for every $x,y$ such that $\Gamma_{x,y}\neq\emptyset$, we have that $\PPP_{x,y}$ is supported on a single, simple curve $\gamma$ (up to reparametrizations of the curve) and $\PPP_{y,x}$ is supported on the same curve (parametrized backwards in time). In particular $\Gamma_{x,y}$ is the image of such curve $\gamma$.

For a traffic plan $\PPP$ we define $\Gamma_\PPP:=\{x:|x|_\PPP>0\}$. We say that a (finite or countable) family of traffic plans $\{\PPP_n\}_{n\in \N}$ is \emph{disjoint} if the measures $\Haus^1 \trace \Gamma_{\PPP_n}$ are mutually singular. 

\begin{definition}[(Connected components)]
Let $\PPP$ be an optimal traffic plan for the mailing problem, i.e. $\PPP \in \OTP(\pi_\PPP)$. Given an open set $U \subset \R^d$, two points $x, y$ are said to be connected in $U$ if there is a chain $x_1 = x, ..., x_l = y$ such that $\emptyset\neq\Gamma_{x_i,x_{i+1}}\subset U$. One can see that connectedness is an equivalence relation on $\Gamma_\PPP$.

We say that $F\subset\Gamma_\PPP$ is a connected component of $\PPP$ in $U$ if $F$ is an equivalence class with respect to such equivalence relation.
\end{definition}

The present section is devoted to show the finiteness of the number of connected components of any optimal traffic plan.
\begin{theorem}\label{rego2}
Let $\alpha>1-\frac{1}{d}$ and $\PPP$ be an optimal traffic plan for the mailing problem, i.e. $\PPP \in \OTP(\pi_\PPP)$, and assume that there exists $C>0$ such that
\begin{equation}
\label{eqn:long curves}
 \int_0^\infty |\dot \gamma| \, dt  \geq C\quad\mbox{ for $\PPP$-a.e. $\gamma$.}
\end{equation}
Then 
\begin{enumerate}
\item
for every $x \in \R^d \setminus (\supp(\mu^-_{\PPP})\cup\supp(\mu^+_{\PPP}))$, $\PPP$ has a finite number of connected components in $\R^d \setminus \{x\}$;
\item the traffic plan $\PPP$ has a finite number of connected components in $\R^d$.
\end{enumerate}
\end{theorem}

\begin{remark}
The hypothesis \eqref{eqn:long curves} is clearly implied by the stronger assumption, recurrent in the literature, that
\begin{equation}
\label{eqn:distant-supports}
d(\supp(\mu^-_{\PPP}),\supp(\mu^+_{\PPP}))>0.
\end{equation}

%
\end{remark}

A similar situation to the one presented in Theorem~\ref{rego2} happens for the irrigation problem (when the coupling between the initial and final measure is not assigned; see \cite[Lemma 8.10 and 8.11]{BCM}) and the general strategies of proof are related in the two cases. On the other hand, in the irrigation problem one can push this result further and show that the traffic plan has the finite tree structure far from the support of the initial and target measures. For the mailing problem, this result is not true, indeed one could build examples where the optimal traffic plan contains cycles (far from the support of the initial and target measures) even when the marginals are given by a finite number of Dirac deltas; this introduces some additional difficulties to the arguments performed below.


In this section we will use the following 
\begin{lemma}\label{lemma_rect_mult} For every $\gamma\in\Lip$ let $M_\gamma$ be a subset of $Im(\gamma)$ (chosen in a measurable fashion). Let $M$ be a $1$-rectifiable set and $\theta\in L^1(\Haus^1\trace M)$ and consider the measure $\nu:=\theta\Haus^1\trace M$. Assume that $\nu$ can be written in the form $\nu=\int\nu_\gamma d\PPP(\gamma)$, where $\nu_\gamma:=\Haus^1\trace M_\gamma$ and $\PPP\in\M(\Lip)$. Then for $\nu$-a.e. $x$ it holds
$$\theta(x)=\PPP(\{\gamma:x\in M_\gamma\}).$$
\end{lemma}
\begin{proof}
We can easily compute for every $E\subset M$
\begin{equation*}
\begin{split}
\int_{E}\theta d\Haus^1&=\nu(E)=\int_{\Lip}\nu_\gamma(E) d\PPP(\gamma)=\int_{\Lip} \int_{M_\gamma\cap E}d\Haus^1 d\PPP(\gamma)\\
&=\int_{E}\int_{\{\gamma:x\in M_\gamma\}} d\PPP(\gamma)d\Haus^1(x)=\int_{E}\PPP(\{\gamma:x\in M_\gamma\})d\Haus^1(x).\qedhere
\end{split}
\end{equation*}  
\end{proof}
The following lemma  associates a canonical traffic plan to every connected component of a given traffic plan. 

\begin{lemma}\label{Lemma plan canonico} Let $\PPP \in \OTP(\pi_\PPP)$ be an optimal traffic plan for the mailing problem, $U$ an open set and let $F$ be a connected component of $\PPP$ in $U$. Then there exists a Borel function $\phi:\gamma \in \Lip\to (\phi_1(\gamma),\phi_2(\gamma)) \in [0,\infty)^2$ such that, denoting $\QQQ:=({\rm{res_{\phi_1(\cdot),\phi_2(\cdot)}}})_\sharp \PPP$, it holds 
\begin{equation}\label{uau}
 \mbox{$|x|_\QQQ=|x|_{\PPP}$ for $\Haus^1$-a.e. $x\in F$ \quad and \quad $|x|_\QQQ=0$ for $\Haus^1$-a.e. $x\not\in F$.}
 \end{equation} 
\end{lemma}
\begin{proof} We observe that we can reduce our analysis to the case when $\PPP$ is supported on (simple) curves parametrized by arc-length. Indeed, to any traffic plan $\PPP$ we can associate another traffic plan $\tilde \PPP$, where each curve in its support is reparametrized by arc-length; $\tilde \PPP$ is also optimal since the $\alpha$-energy is invariant under reparametrization. Assuming that the statement holds for $\tilde \PPP$, giving a map $\tilde \phi$, we obtain $\phi$ simply setting, for $i=1,2$ 
$$\phi_i(\gamma):=\gamma^{-1}(\tilde\gamma(\tilde\phi_i(\tilde\gamma))),$$
where $\tilde\gamma$ is the arch-length parametrization of $\gamma$ and we denoted 
$$\gamma^{-1}(x):=\inf\{t\geq 0:\gamma(s)\neq x, \forall s\leq t\}.$$
 
{\it Step 1.} 
 For every $\gamma\in\Lip$, we define the "bad set" of $\gamma$ as
$$B_F(\gamma):=\{(s,t)\in[0,T(\gamma)]^2:\gamma(s)\in F, \gamma(t)\in F, {\rm{ and }}\; \exists r\in(s,t):\gamma(r)\not\in F\}.$$

We denote the set of \rm{bad curves} 
$$B_F:=\{\gamma\in\Lip:\Leb^2(B_F(\gamma))>0\}.$$
We claim that, if $\PPP(B_F)>0$, then there exist two points $x,y\in F$ and a set $B\subset B_F$ with $\PPP(B)>0$, such that for every $\gamma\in B$ there exist $(t_1,t_2)\in B_F(\gamma)$ with $\gamma(t_1)=x$, $\gamma(t_2)=y$ (note that $t_1$ and $t_2$ may depend on $\gamma$).

%
To show this claim, assume that $\PPP(B_F)>0$ and let $\mu$ and $\nu$ be the positive measures
\begin{equation}\label{beee}
\mu:=\int_{\Lip} \Haus^1\trace(\rm{Im}(\gamma)\cap F)d\PPP(\gamma),
\end{equation}
\begin{equation}\label{mutre}
\nu(E):=\int_{B_F} \Haus^1(E\cap\{\gamma(t):t\in\eta_1(B_F(\gamma))\})d\PPP(\gamma), \quad\mbox{ for every $E$ Borel},
\end{equation}
where we denoted with $\eta_1$ the projection from $\R\times \R$ onto the first component. 
Since $B_F$ has positive measure, then by Fubini's theorem $\nu(\R^d)>0$. By Theorem \ref{Theoret}, $\mu$ is absolutely continuous with respect to $\Haus^1$ and is supported on a rectifiable set. Since clearly $\nu\leq\mu$, then $\nu$ is a rectifiable measure, i.e.
$$\nu=\theta \Haus^1\trace M, \qquad \mbox{where $\theta>0$ and $M$ is $1$-rectifiable}.$$ By Lemma \ref{lemma_rect_mult}, for $\Haus^1$-a.e. point $x\in M$, there exists a subset $B'$ of $B_F$ of positive $\PPP$-measure such that for every curve $\gamma\in B'$ it holds $x\in\{\gamma(t):t\in\eta_1(B_F(\gamma))\}$. In particular $x\in F$.  
Let $\rho$ be the positive measure defined by
\begin{equation}\label{muquattro}
\rho(E):=\int_{B'} \Haus^1(E\cap\{\gamma(t):t\in\eta_2(B_F(\gamma))\})d\PPP(\gamma), \quad\mbox{ for every $E$ Borel},
\end{equation}
where $\eta_2$ is the projection from $\R\times \R$ onto the second component. 
With a similar argument we prove that there exist a point $y\neq x$, $y\in F$ of positive multiplicity for $\rho$, and a subset $B\subset B'$ of positive measure such that $y\in\{\gamma(t):t\in\eta_2(B_F(\gamma))\}$, for every $\gamma\in B$. 

{\it{Step 2}.}We prove that $\PPP(B_F)=0$ and that for $\PPP$-a.e. $\gamma\in\Lip$
\begin{equation}
\label{eqn:represent}
\{t\in[0,T(\gamma)]:\gamma(t)\in F\}=I(\gamma)\cup E(\gamma),
\end{equation}
where $I(\gamma)$ is an open interval (possibly trivial) and $E(\gamma)$ has measure zero.

By contradiction, assume that  $\PPP(B_F)>0$ and let $x,y$ be two points as in the claim proved in the previous step. By the single path property of optimal traffic plans, almost all fibers coincide between $x$ and $y$ with a unique curve $\gamma_0$ (or with the same curve $\gamma_0$ parametrized backwards in time). In particular there exists a point $z\in {\rm{Im}}(\gamma_0)\setminus F$. On the other hand, by definition of connected component, $z$ belongs to the same connected component of $x$ and $y$, hence $z\in F$, which is a contradiction.

As a consequence of the fact that $\PPP(B_F)=0$, we deduce \eqref{eqn:represent} from the definition of $B_F(\gamma)$.

{\it Step 3}. We conclude the proof. Consider the function 
$$\phi(\gamma)=
\begin{cases}
(\phi_1(\gamma),\phi_2(\gamma)), &{\mbox{ if $I(\gamma)=(\phi_1(\gamma),\phi_2(\gamma))$ is defined and non trivial}},\\
(0,0), &\mbox{ otherwise}.
\end{cases}$$
By Step 2, it holds
\begin{equation}\label{muuu}
\mu=\int_{\Lip} \Haus^1\trace(\rm{Im}({\rm{res}}_{\phi(\gamma)}(\gamma)))d\PPP=\int_{\Lip} \Haus^1\trace({\rm{Im}}(\gamma))d\QQQ(\gamma).
\end{equation}
By Lemma \ref{lemma_rect_mult} and the definition of multiplicity, the conclusion of Lemma \ref{Lemma plan canonico} follows.\end{proof}

\begin{remark}\label{disj}
By Lemma \ref{Lemma plan canonico}, we know that the multiplicity of different connected components of an optimal traffic plan $\PPP$ with respect to an open set $U$ is nonzero on disjoint sets, namely $|x|_{\PPP'} > 0$ implies that $|x|_{\PPP''} = 0$ for all $\PPP'$ and $\PPP''$ traffic plans associated to different connected components. Hence, any countable subfamily of connected components 
is made by disjoint traffic plans. 
\end{remark}

\begin{remark}(Optimal plans have at most countably many connected components.)\label{energpos} Given an optimal traffic plan $\PPP$ and $U\subset \R^d$ open, we claim that, under assumption \eqref{eqn:long curves}, the traffic plan canonically associated to each connected component of $\PPP$ in $U$ has strictly positive energy. To prove it, we consider a connected component $F$ of $\PPP$ and a point $x_0 \in F$. We define the traffic plan
$$\PPP':= \PPP \trace \{ \gamma \in \Lip : x_0 \in \text{Im(}\gamma\text{)} \}.$$
We construct a new traffic plan $\PPP''$ restricting each simple curve passing through $x_0$ to the maximal interval of times containing $t_{x_0}$ where the image of the curve is in $U$. By \eqref{eqn:long curves} and \eqref{energy}, we have that 
$$\E(\PPP'') \geq \min\{C, \text{dist}(x_0, \partial U) \}.$$
Since $\PPP'$ is simple path, by Remark \eqref{simple} we deduce that $\MM(\PPP'')>0$. This implies the existence of a point $ y\neq x_0$ such that $|y|_{\PPP''}>0$. By the single path property, we conclude that $y \in F$ and since the connected components are arc-wise connected by definition, this implies the claim.
From the finiteness of the energy of $\PPP$ and the fact that all connected components are disjoint (as observed in Remark \ref{disj}) and carry strictly positive energy, we deduce that there can only be at most countably many connected components.
\end{remark}

\begin{remark}Let $\QQQ$ be the traffic plan canonically associated to a connected component $F$ of an optimal traffic plan $\PPP$ through Lemma \ref{Lemma plan canonico}, i.e. $\QQQ:=({\rm{res_{\phi_1(\cdot),\phi_2(\cdot)}}})_\sharp \PPP$. We can compute
\begin{equation*}
\begin{split}
\int_{\Lip}\int_{\gamma(t) \in F} &|\gamma(t)|^{\alpha-1}  |\dot \gamma(t)|d\PPP =\int_{\Lip}\int_{\gamma(t) \in F} |[res_{\phi_1(\cdot),\phi_2(\cdot)}(\gamma)](t)|^{\alpha-1}  |[\frac{d}{dt} [res_{\phi_1(\cdot),\phi_2(\cdot)}(\gamma)](t)|d\PPP\\
&=\int_{\Lip}\int_0^\infty |\gamma(t)|^{\alpha-1}  |\dot \gamma(t)|d\QQQ=\E(\QQQ)
\end{split}
\end{equation*} 
\end{remark}


\begin{lemma}\label{primo}
Let $\PPP \in \OTP(\pi_\PPP)$ be an optimal traffic plan for the mailing problem, satisfying \eqref{eqn:long curves}, $U\subset \R^d$ be an open set such that $\Haus^1(U^c)=0$  and  $\{F_n\}_n$ be the (finite or countable) family of connected components of $\PPP$ in $U$. Then the traffic plans $\{\PPP_n\}_n$ canonically associated to $\{F_n\}_n$ through Lemma \ref{Lemma plan canonico} are optimal, disjoint traffic plans. Moreover for every couple $(n,m) \in \N^2$, the traffic plan $\PPP_n+\PPP_m$ is optimal.
\end{lemma}
The idea of the proof goes as follows. Assume that the traffic plan $\PPP_1$ associated to a connected component $F$ of $\PPP$ is not optimal and let $\tilde \PPP_1$ be an optimizer with the same coupling as $\PPP_1$. Remember that $\PPP$-a.e. curve $\gamma$ intersects $F_1$ for an open interval of times $(\phi_1(\gamma),\phi_2(\gamma))$, plus a set of measure zero. We construct a competitor for $\PPP$ by ``sewing'' the two pieces of each curve $\gamma$ corresponding to the complementary of $(\phi_1(\gamma),\phi_2(\gamma))$ with a new curve having the same end-points as the restriction of $\gamma$ to $(\phi_1(\gamma),\phi_2(\gamma))$. This new curve is the unique connection between $\gamma(\phi_1(\gamma))$ and $\gamma(\phi_2(\gamma))$ in $\tilde \PPP_1$.
\begin{proof}
The traffic plans $\{\PPP_n\}_n$ are disjoint, thanks to Remark \ref{disj}. Since $\{\PPP_n\}_n$ and $\PPP$ are simple path and $\Haus^1(U^c)=0$,
\begin{equation}\label{wow}
\begin{split}
\sum_n\E(\PPP_n)&=\sum_n\MM(\PPP_n)=\sum_n\int_{F_n}|x|_{\PPP_n}^\alpha d\Haus^1(x)\\
&=\sum_n\int_{F_n}|x|_{\PPP}^\alpha d\Haus^1(x)=\int_{U^c \cup (\cup_nF_n)}|x|_{\PPP}^\alpha d\Haus^1(x)=\MM(\PPP)=\E(\PPP).
\end{split}
\end{equation}
We need to show that they are optimal with respect to their coupling.
Assume by contradiction that there exists one of them which is not optimal for the associated coupling. Up to reorder the sequence, we can assume $\PPP_1$ is not optimal. 

We consider the coupling $\pi:=(e_0,e_\infty)_\#\PPP_1$ and we disintegrate $\PPP_1$ with respect to $\pi$:
$$\PPP_1=\pi(x,y)\otimes \PPP_1^{x,y}.$$
We observe that, since $\PPP$ is optimal, then it satisfies the simple path property and the single path property.
This implies that also $\PPP_1$  satisfies the simple path property and the single path property  (even if it is not optimal) and in particular $\E(\PPP_1)=\MM(\PPP_1)$  (see Remark \ref{simple}). From the single path property we also deduce that for $\pi$-a.e. $(x,y)$ there exists $\gamma^{xy} \in \Lip$ such that (up to reparametrization of the curves) 
$ \PPP_1^{x,y}= \delta_{\gamma^{xy}}$.

Since by contradiction $\PPP_1 \not \in \OTP (\pi)$, then there exists $\tilde \PPP_1 \in \OTP (\pi)$ (which without loss of generality can be assumed supported on curves parametrized by arc-length) and by Proposition \ref{massaenergia} we deduce
\begin{equation}\label{ugual}
\MM(\tilde \PPP_1)=\E(\tilde \PPP_1)<\E(\PPP_1)=\MM(\PPP_1).
\end{equation}

Since $\tilde\PPP_1$ is optimal, it satisfies the single path property. Hence, for $\pi$-a.e. $(x,y)$, there exists  $\tilde\gamma^{xy} \in \Lip$ such that $\tilde\gamma^{xx}$ is the curve constantly equals to $x$ and the disintegration of $\tilde\PPP_1$ with respect to $(e_0,e_\infty)$ reads 
\begin{equation}\label{un}
\tilde\PPP_1=\pi(x,y)\otimes \delta_{\tilde\gamma^{xy}}.
\end{equation}
Let $(\phi_1,\phi_2)$ the map associated to $\PPP_1$ through Lemma \ref{Lemma plan canonico}.  
We consider the following map $\xi: \gamma \in \Lip \to \tilde\gamma^{\gamma(\phi_1(\gamma)),\gamma(\phi_2(\gamma))} \in \Lip$, which is well defined for $\PPP$-a.e. $\gamma$.
Testing \eqref{un} with a general test function and since $\xi({\rm{res_{\phi_1(\cdot),\phi_2(\cdot)}}}(\cdot))= \xi(\cdot)$, we deduce that
\begin{equation}\label{ciaoamici}
\tilde\PPP_1=\xi_{\#}\PPP = \xi_{\#}\PPP_1.
\end{equation}
We build now the following map $\varphi:\Lip \to \Lip$

\begin{equation*}
[\varphi(\gamma)](t):=
\begin{cases}
\gamma(t) \qquad &\mbox{if } t\leq \phi_1(\gamma),
\\
[\xi(\gamma)](t-\phi_1(\gamma)) \qquad &\mbox{if }  \phi_1(\gamma)<t<T\left(\xi(\gamma)\right)+ \phi_1(\gamma),
\\
\gamma(t+\phi_2(\gamma)- T\left(\xi(\gamma)\right)- \phi_1(\gamma)) \qquad &\mbox{if } t\geq T\left(\xi(\gamma)\right)+ \phi_1(\gamma).
\end{cases}
\end{equation*}
Since $e_0(\gamma)=e_0(\varphi(\gamma))$ and $e_\infty(\gamma)=e_\infty(\varphi(\gamma))$ for every $\gamma \in \Lip$, we deduce that $\tilde \PPP:=\varphi_\#\PPP\in \TP(\pi_{\PPP})$. 
{
 Consequently we can compute for $\Haus^1$-a.e. $z \in \{|z|_{\tilde \PPP_1}>0\}$
\begin{equation}\label{:)eccoci}
\begin{split}
|z|_{\tilde \PPP_1}&=\tilde \PPP_1(\{\eta \in \Lip | z \in Im (\eta)\})= \tilde \PPP_1(\{ \xi(\gamma)\in \Lip : z \in Im ( \xi(\gamma))  \})\\
&= \PPP_1(\{({\rm{res_{\phi_1(\gamma),\phi_2(\gamma)}}})(\gamma) \in \Lip | z \in Im (\xi(\gamma)) \})\\
&= \PPP({\rm{res_{\phi_1(\cdot),\phi_2(\cdot)}}}^{-1}(\{({\rm{res_{\phi_1(\gamma),\phi_2(\gamma)}}})(\gamma) \in \Lip | z \in Im (\xi(\gamma)) \}))\\
&= \PPP(\{\gamma \in \Lip :   z \in Im (\xi(\gamma)) \}))) \leq \tilde \PPP(\varphi(\{\gamma \in \Lip :   z \in Im (\xi(\gamma)) \}))\\
&\leq \tilde \PPP(\{\varphi(\gamma) \in \Lip | z \in Im (\varphi(\gamma)) \})=\tilde \PPP(\{\eta \in \Lip |  z \in Im (\eta) \})=|z|_{\tilde \PPP},
\end{split}
\end{equation}
where line one is by definition of multiplicity and by the fact that $\tilde \PPP_1$ is supported on $\xi(\Lip)$; line two is by \eqref{ciaoamici} and  $\xi({\rm{res_{\phi_1(\cdot),\phi_2(\cdot)}}}(\cdot))= \xi(\cdot)$; 
line three is due to $\PPP_1:=({\rm{res_{\phi_1(\cdot),\phi_2(\cdot)}}})_\sharp \PPP$;  the inequality in line four is a consequence of $\tilde \PPP= \varphi_\# \PPP$; line five is by the inclusion $Im (\xi(\gamma)) \subseteq Im (\varphi(\gamma)) $.}

Given $n \in \N$, if we denote for every $\gamma \in \Lip$ the maximal interval with respect to $F_n$ as $(\psi_1(\gamma),\psi_2(\gamma))$, we can consequently compute
\begin{equation}\label{:)arieccoci}
\begin{split}
|z|_{ \PPP_n}&=\PPP_n(\{\gamma \in \Lip | z \in Im (\gamma)\})=  \PPP(res_{\psi_1(\cdot),\psi_2(\cdot)}^{-1}(\{\gamma \in \Lip | z \in Im (\gamma)\}))\\
&\leq  \tilde \PPP(\varphi (res_{\psi_1(\cdot),\psi_2(\cdot)}^{-1}(\{\gamma \in \Lip | z \in Im (\gamma)\})))\leq\tilde \PPP(\{\eta \in \Lip |  z \in Im (\eta) \})=|z|_{\tilde \PPP},
\end{split}
\end{equation}
where the second equality follows by Lemma \ref{Lemma plan canonico} (construction of $\PPP_n=(res_{\psi_1(\cdot),\psi_2(\cdot)})_{\#}(\PPP)$), the first inequality is by definition of $\tilde \PPP=\varphi_\# \PPP$ and the second inequality is by set inclusion, since $\varphi (res_{\psi_1(\cdot),\psi_2(\cdot)}^{-1}(\{\gamma \in \Lip | z \in Im (\gamma)\})) \subset \{\eta \in \Lip |  z \in Im (\eta) \}$.
By Remark \ref{disj}, for $\Haus^1$-a.e. $z$ there exists at most one $n>1$ such that $|z|_{ \PPP_n}>0$. Summarizing this information with \eqref{:)eccoci}, \eqref{:)arieccoci}, we get 
\begin{equation*}
\max\{ |z|_{\tilde \PPP_1}, \sum_{n>1}|z|_{\PPP_n}\} = \max\{ |z|_{\tilde \PPP_1}, \max_{n>1}|z|_{\PPP_n}\} \leq |z|_{\tilde \PPP} , 
\end{equation*}
from which we compute for $\PPP$-a.e. curve $\gamma$
\begin{equation*}
\begin{split}
&\int_0^\infty {|\varphi(\gamma)(t)|_{\tilde \PPP}^{\alpha-1}} \Big|\frac{d}{dt} \varphi(\gamma)(t) \Big| \, dt\\
&=\int_0^\infty {|[\xi(\gamma)](t)|_{\tilde \PPP}^{\alpha-1}}\Big|\frac{d}{dt}[\xi(\gamma)](t) \Big|  \, dt + \int_0^{\phi_1(\gamma)} {|\gamma(t) |_{\tilde \PPP}^{\alpha-1}} |\dot \gamma(t) | \, dt
+
\int_{\phi_2(\gamma)}^\infty {|\gamma(t) |_{\tilde \PPP}^{\alpha-1}} |\dot \gamma(t) | \, dt
\\&{\leq}\int_0^\infty {|[\xi(\gamma)](t)|_{\tilde \PPP_1}^{\alpha-1}} \Big|\frac{d}{dt}[\xi(\gamma)](t) \Big| \, dt
 +
\sum_{n>1}\left(\int_0^{\phi_1(\gamma)} {|\gamma(t) |_{ \PPP_n}^{\alpha-1}} |\dot \gamma(t) | \, dt
+
\int_{\phi_2(\gamma)}^\infty {|\gamma(t) |_{\PPP_n}^{\alpha-1}} |\dot \gamma(t) | \, dt\right)
\\&{\leq}\int_0^\infty {|[\xi(\gamma)](t)|_{\tilde \PPP_1}^{\alpha-1}} \Big|\frac{d}{dt}[\xi(\gamma)](t) \Big| \, dt  +
\sum_{n>1} \int_{\gamma(t) \in F_n} {|\gamma(t) |_{ \PPP_n}^{\alpha-1}} |\dot \gamma(t) | \, dt.
\end{split}
\end{equation*}

Integrating with respect to $\PPP$, we get that
\begin{equation*}
\begin{split}
\E(\tilde \PPP)  & \, \, \, \,   =\int \int_0^\infty {|\gamma(t)|_{\tilde \PPP}^{\alpha-1}} |\frac{d}{dt}\gamma(t) | \, dtd\tilde \PPP(\gamma)= \int \int_0^\infty {|\varphi(\gamma)(t)|_{\tilde \PPP}^{\alpha-1}} |\frac{d}{dt} \varphi(\gamma)(t) | \, dtd\PPP(\gamma)
\\& \, \, \, \, \, {\leq}\int\int_0^\infty {|[\xi(\gamma)](t)|_{\tilde \PPP_1}^{\alpha-1}} \Big|\frac{d}{dt}[\xi(\gamma)](t) \Big| \, dtd\PPP(\gamma)
+ \sum_{n>1} \int \int_{\gamma(t) \in F_n} {|\gamma(t) |_{ \PPP_n}^{\alpha-1}} |\dot \gamma(t) | \, dtd\PPP(\gamma)
\\
&\overset{\eqref{ciaoamici}}{=} \E(\tilde \PPP_1)+
\sum_{n>1} \E(\PPP_n) < \E(\PPP_1)+ \sum_{n>1} \E(\PPP_n)\overset{\eqref{wow}}{=}\E(\PPP),
\end{split}
\end{equation*}
The last inequality  contradicts the optimality of $\PPP\in \OTP(\pi_{\PPP})$.
With the same proof, namely replacing $\PPP_1$ with $\PPP_n+ \PPP_m$ in the argument above, one can also show that, for every couple $(n,m) \in \N^2$, the traffic plan $\PPP_n+\PPP_m$ is optimal.
\end{proof}
\begin{lemma}\label{secondo}
Let $\alpha>1-\frac{1}{d}$ and $(\PPP_n)_{n=1}^\infty$ a sequence of disjoint traffic plans such that $\sum_{n\in \N}|\PPP_n|<+\infty$. If 
\eqref{eqn:long curves} is in force for some $C>0$,
then there exists $n\in \N$ or a couple $(n,m)\in \N^2$ such that either $\PPP_n$ or $\PPP_n+ \PPP_m$ is not optimal, i.e. either $\PPP_n \not \in \OTP(\pi_{\PPP_n})$ or $\PPP_n+ \PPP_m \not \in \OTP(\pi_{\PPP_n+ \PPP_m})$. \end{lemma}
\begin{proof} Without loss of generality, we can assume that all traffic plans are supported on curves parametrized by arc-length. {Indeed, assume we are able to prove Lemma \ref{secondo} in case the elements of the sequence are supported on curves parametrized by arc-length. If $(\PPP_n)_{n=1}^\infty$ does not satisfy this assumption, since \eqref{eqn:long curves} is invariant under reparametrization, we can apply Lemma \ref{secondo} to a sequence $(\tilde \PPP_n)_{n=1}^\infty$, obtained reparametrizing by arc-length the curves on which the elements of $(\PPP_n)_{n=1}^\infty$ are supported. Then there exists $n\in \N$ or a couple $(n,m)\in \N^2$ such that either $\tilde \PPP_n \not \in \OTP(\pi_{\tilde \PPP_n})$ or $\tilde \PPP_n+ \tilde \PPP_m \not \in \OTP(\pi_{\tilde \PPP_n+ \tilde \PPP_m})$. But, since 
$$\E(\tilde \PPP_n)=\E( \PPP_n), \quad \E(\tilde \PPP_n+ \tilde \PPP_m)=\E( \PPP_n + \PPP_m), \quad \pi_{\tilde \PPP_n}=\pi_{\PPP_n}, \quad \pi_{\tilde \PPP_n+ \tilde \PPP_m}=\pi_{\PPP_n+ \PPP_m},$$
then  either $\PPP_n \not \in \OTP(\pi_{\PPP_n})$ or $\PPP_n+ \PPP_m \not \in \OTP(\pi_{\PPP_n+ \PPP_m})$.}

Now we prove Lemma \ref{secondo} in case the elements of $(\PPP_n)_{n=1}^\infty$ are supported on curves parametrized by arc-length. If there exists $n\in \N$ such that  $\PPP_n$ is not optimal, the statement is true; we can consequently assume that all $\PPP_n$ are optimal.
We first renormalize every plan $\PPP_n$, defining
\begin{equation}\label{defren}
\PPP_n':=\frac{\PPP_n}{|\PPP_n|},
\end{equation}
which is an optimal traffic plan of unit mass supported on curves parametrized by arc-length. Moreover, by Remark \ref{semtpc}, there exists $C>0$ such that $\PPP_n' \in \TP_C$ for every $n \in \N$. 
Up to extracting a subsequence, there exists $\PPP'\in \PP(\Lip)$ such that
$$\PPP_n'\rightharpoonup \PPP'.$$
Taking the push-forward through the maps $(e_0,e_\infty)$, $e_0$ and $e_\infty$, we deduce respectively that
$\pi_{\PPP_n'}=\frac{\pi_{\PPP_n}}{|\PPP_n|}\rightharpoonup \pi_{\PPP'}$ and $\mu^\pm_{\PPP_n'}=\frac{\mu^\pm_{\PPP_n}}{|\PPP_n|}\rightharpoonup \mu^\pm_{\PPP'}$.
We can apply Theorem \ref{thm:good-version} to deduce that, for every $\e>0$, there exists $n_0\in \N$ such that for every $n,m\geq n_0$, there exist two traffic plans $\PPP^1_{n,m} \in \TP(\mu^-_{\PPP_n'}, \mu^-_{\PPP_m'})$, $\PPP^2_{n,m}\in \TP(\mu^+_{\PPP_m'}, \mu^+_{\PPP_n'})$ such that
\begin{equation}\label{111}	
	 \E (\PPP^1_{n,m}) \leq \e, \qquad \E (\PPP^2_{n,m}) \leq \e
	 \end{equation}
	and there exists a traffic plan obtained as a concatenation between $\PPP_m'$ and $\PPP^2_{n,m}$, and a further concatenation $\PPP''$ between $\PPP^1_{n,m}$ and the latter concatenation such that $\PPP'' \in \TP(\pi_{\PPP_n'})$.
	
We claim that for $n$ large enough, the following estimates hold
\begin{equation}\label{1}
\E(\PPP_n)\geq C|\PPP_n|^\alpha
\end{equation}
and for every $\e>0$ to be chosen later, up to choose $n$ and $m$ even larger (depending on $\e$)
\begin{equation}\label{2}
\E(\PPP_n)\geq \frac{|\PPP_n|^\alpha}{|\PPP_m|^\alpha}\E(\PPP_m)-\e|\PPP_n|^\alpha.
\end{equation}
Indeed, since by assumption the traffic plans $\PPP_n$ are optimal, then $\PPP_n'\in \OTP(\pi_{\PPP_n'})$. Moreover, thanks to \eqref{eqn:long curves} for every $n \in \N$ we have that $\PPP_n$-a.e. (and hence $\PPP'_n$-a.e.) $\gamma$ has length uniformly bounded from below by $C$. 
We compute
$$ \E(\PPP_n')=\int_{\Lip} \int_{\R^+} |\gamma( t)|_{\PPP_n'}^{\alpha-1}|\dot \gamma(t)| dt\, d\PPP_n'(\gamma)\geq \liminf_{n\to \infty}\int_{\Lip} T(\gamma)\, d\PPP_n'(\gamma) \geq C
, \qquad \forall n \in \N,$$
 and renormalizing the traffic plans we deduce $\E(\PPP_n)=\E(\PPP_n')|\PPP_n|^\alpha \geq C|\PPP_n|^\alpha$, namely \eqref{1}.
In order to prove \eqref{2}, we observe that Theorem \ref{thm:main} implies that the limit $\PPP'$ of the optimal sequence $(\PPP_n')_{n=1}^\infty$ is optimal and that $ \E(\PPP') = \lim_{n\to \infty}\E(\PPP_n')$. In particular $\E(\PPP_n')$ is a Cauchy sequence, that is for every $\e>0$ we have 
\begin{equation}\label{3}
\e> |\E(\PPP_n') - \E(\PPP_m')| = \left| \frac{\E(\PPP_n)}{|\PPP_n|^\alpha}-\frac{\E(\PPP_m)}{|\PPP_m|^\alpha}\right|,
\end{equation}
for every $m, n$ sufficiently large (the latter equality follows by the definition \eqref{defren} of $\PPP_n'$). 
This gives in particular \eqref{2}.

Now we prove that $\PPP_n+ \PPP_m$ is not optimal, because we can construct a better competitor
$$\PPP''|\PPP_n|+\PPP'_m|\PPP_m| \in \TP(\pi_{\PPP_n+ \PPP_m}).$$
Indeed, by Lemma \ref{lemma:all-you-want-on-conc} (we think the traffic plan as a concatenation between $\PPP^1_{n,m}$, $(|\PPP_n|+|\PPP_m|)\PPP'_m$ and $\PPP^2_{n,m}$), we can estimate its energy as follows:
\begin{equation}\label{prima}
\begin{split}
\E(\PPP''|\PPP_n|+\PPP'_m|\PPP_m|)&\leq \E(\PPP^1_{n,m}|\PPP_n|)+\E(\PPP^2_{n,m}|\PPP_n|)+\E((|\PPP_n|+|\PPP_m|)\PPP'_m)\\
&\overset{\eqref{111}}{\leq} 2\e|\PPP_n|^\alpha+\E\left (\left (1+\frac{|\PPP_n|}{|\PPP_m|}\right )\PPP_m\right )\\
&= 2\e|\PPP_n|^\alpha+\E(\PPP_m)\left (1+\frac{|\PPP_n|}{|\PPP_m|}\right )^\alpha\\
&\leq 2\e|\PPP_n|^\alpha+\E(\PPP_m)\left (1+\alpha \frac{|\PPP_n|}{|\PPP_m|}\right ).
\end{split}
\end{equation}
On the other hand, since $\PPP_n$ and $\PPP_m$ are disjoint and optimal traffic plans, by Proposition \ref{massaenergia} we compute 
\begin{equation}\label{seconda}
\begin{split}
\E(\PPP_n+\PPP_m)&\geq \MM(\PPP_n+\PPP_m)= \MM(\PPP_n)+\MM(\PPP_m)=\E(\PPP_n)+\E(\PPP_m)\\
&\overset{\eqref{2}}{\geq} \E(\PPP_m)+\frac{|\PPP_n|^\alpha}{|\PPP_m|^\alpha}\E(\PPP_m)-\e|\PPP_n|^\alpha.
\end{split}
\end{equation}
Combining the previous inequalities \eqref{prima} and \eqref{seconda} we have
\begin{equation}\label{final-ole}
\begin{split}
\E(\PPP''&|\PPP_n|+\PPP'_m|\PPP_m|)-\E(\PPP_n+\PPP_m)\leq  \left(\alpha \frac{|\PPP_n|}{|\PPP_m|}-\frac{|\PPP_n|^\alpha}{|\PPP_m|^\alpha}\right )\E(\PPP_m)+3\e|\PPP_n|^\alpha\\
&\overset{\eqref{1}}{\leq} \left(\alpha \frac{|\PPP_n|}{|\PPP_m|}-\frac{|\PPP_n|^\alpha}{|\PPP_m|^\alpha}\right)\E(\PPP_m)+\frac 3C\e\frac{|\PPP_n|^\alpha}{|\PPP_m|^\alpha}\E(\PPP_m)\\
&\leq \left(\alpha \frac{|\PPP_n|}{|\PPP_m|}-\left (1-\frac 3C\e \right)\frac{|\PPP_n|^\alpha}{|\PPP_m|^\alpha}\right)\E(\PPP_m)<0,
\end{split}
\end{equation}
where in the last inequality we chose $\varepsilon$ small enough to get $1-\frac 3C\e>\alpha$ and $n$ bigger enough than $m$ in order to guarantee that $\frac{|\PPP_n|}{|\PPP_m|}<1$ and consequently $\frac{|\PPP_n|}{|\PPP_m|}<\frac{|\PPP_n|^\alpha}{|\PPP_m|^\alpha}$.
We deduce from inequality \eqref{final-ole} that $\PPP_n+ \PPP_m$ is not optimal.
\end{proof}

\begin{proof}[Proof of Theorem~\ref{rego2}]
Assume by contradiction that  the sequence $(\PPP_n)_{n=1}^\infty$ of the traffic plans canonically associated to the connected components of $\PPP$ in $\R^d \setminus \{x\}$ (or in $\R^d$) as shown in Lemma \ref{Lemma plan canonico} is infinite. Since \eqref{eqn:long curves} is in force, we are in the hypothesis to apply Lemma \ref{secondo}, hence there exists $n\in \N$ or a couple $(n,m)\in \N^2$ such that $\PPP_n$ or $\PPP_n+ \PPP_m$ is not optimal. On the other hand, by Lemma \ref{primo}, we know that this is not possible. 
\end{proof}

\bibliography{references3}
\bibliographystyle{is-alpha}

%
%

\vskip .5 cm

{\parindent = 0 pt\begin{footnotesize}

M.C. 
\\
Institute for Theoretical Studies, ETH Z\"urich, Clausiusstrasse 47, CH-8092 Z\"urich, Switzerland
\\
e-mail M.C.: {\tt maria.colombo@eth-its.ethz.ch}
\\
~
\\
A.D.R.
\\
Courant Institute of Mathematical Sciences, New York University, New York, NY, USA
\\
e-mail A.D.R.: {\tt derosa@cims.nyu.edu}
\\
~
\\
A.M.
\\
Institut f\"ur Mathematik, Universit\"at Z\"urich, Z\"urich, Switzerland
\\
e-mail A.M.: {\tt andrea.marchese@math.uzh.ch}
\end{footnotesize}
}

\end{document}